\documentclass{amsart}

\usepackage{graphicx}
\usepackage{geometry}
\usepackage{color}
\usepackage{esint}
\usepackage{verbatim}

\usepackage[colorlinks=true,urlcolor=blue,citecolor=red,linkcolor=blue,linktocpage,pdfpagelabels,bookmarksnumbered,bookmarksopen]{hyperref}

\usepackage[hyperpageref]{backref}

\numberwithin{equation}{section}

\theoremstyle{plain}
\newtheorem{teo}{Theorem}[section]
\newtheorem{lemma}[teo]{Lemma}
\newtheorem{prop}[teo]{Proposition}
\newtheorem{coro}[teo]{Corollary}
\theoremstyle{definition}
\newtheorem{oss}[teo]{Remark}

\newtheorem*{ack}{Akcnowledgments}

\title{The fractional Makai-Hayman inequality}
\author[Bianchi]{Francesca Bianchi}
\address[F.\ Bianchi]{Dipartimento di Scienze Matematiche, Fisiche e Informatiche
	\newline\indent
	Universit\`a di Parma
	\newline\indent
	Parco Area delle Scienze 53/a, Campus, 43124 Parma, Italy}
\email{francesca.bianchi@unipr.it}

\author[Brasco]{Lorenzo Brasco}
\address[L.\ Brasco]{Dipartimento di Matematica e Informatica
	\newline\indent
	Universit\`a degli Studi di Ferrara
	\newline\indent
	Via Machiavelli 35, 44121 Ferrara, Italy}
\email{lorenzo.brasco@unife.it}

\subjclass[2010]{47A75, 39B72, 35R11}
\keywords{Poincar\'e inequality, fractional Laplacian, inradius, simply connected sets, Cheeger inequality.}

\date{\today}

\begin{document}

\begin{abstract}
We prove that the first eigenvalue of the fractional Dirichlet-Laplacian of order $s$ on a simply connected set of the plane can be bounded from below in terms of its inradius only. This is valid for $1/2<s<1$ and we show that this condition is sharp, i.\,e. for $0<s\le 1/2$ such a lower bound is not possible. The constant appearing in the estimate has the correct asymptotic behaviour with respect to $s$, as it permits to recover a classical result by Makai and Hayman in the limit $s\nearrow 1$. The paper is as self-contained as possible.
\end{abstract}

\maketitle

\begin{center}
\begin{minipage}{11cm}
\small
\tableofcontents
\end{minipage}
\end{center}

\section{Introduction}

\subsection{Background}
For an open set $\Omega\subset\mathbb{R}^N$, we indicate by $W^{1,2}_0(\Omega)$ the closure of $C^\infty_0(\Omega)$ in the Sobolev space $W^{1,2}(\Omega)$. We then consider the following quantity
\[
\lambda_1(\Omega):=\inf_{u\in W^{1,2}_0(\Omega)\setminus\{0\}}\frac{\displaystyle\int_\Omega|\nabla u|^2\,dx}{\displaystyle\int_\Omega|u|^2\,dx},
\]
which coincides with the bottom of the spectrum of the Dirichlet-Laplacian on $\Omega$. Observe that for a general open set, such a spectrum may not be discrete and the infimum value $\lambda_1(\Omega)$ may not be attained. Whenever a minimizer $u_1\in W^{1,2}_0(\Omega)$ of the problem above exists, we call $\lambda_1(\Omega)$ the {\it first eigenvalue of the Dirichlet-Laplacian on} $\Omega$. 
\par
By definition, such a quantity is different from zero if and only if $\Omega$ supports the Poincar\'e inequality
\[
c\,\int_\Omega |u|^2\,dx\le \int_\Omega |\nabla u|^2\,dx,\qquad \mbox{ for every } u\in C^\infty_0(\Omega).
\]
It is well-known that this happens for example if $\Omega$ is bounded or with finite measure or even bounded in one direction only.
However, in general it is quite complicate to give more general geometric conditions, assuring positivity of $\lambda_1$. In this paper, we will deal with the two-dimensional case $N=2$.  
\par
In this case, there is a by now classical result which asserts that 
\begin{equation}
\label{MH}
\lambda_1(\Omega)\ge\frac{C}{r_\Omega^2},
\end{equation}
for every {\it simply connected} set $\Omega\subset\mathbb{R}^2$.
Here $C>0$ is a universal constant and the geometric quantity $r_\Omega$ is the {\it inradius} of $\Omega$, i.e. the radius of a largest disk contained in $\Omega$. More precisely, this is given by
\[
r_\Omega:=\sup\Big\{\rho>0\,:\,\exists\, x\in\Omega\mbox{ such that }B_\rho(x)\subset\Omega\Big\}.
\]
Inequality \eqref{MH} is of course in scale invariant form, by recalling that $\lambda_1$ scales like a length to the power $-2$, under dilations. 
\par
Such a result is originally due to Makai (see \cite[equation (5)]{Ma}). It implies in particular that for a simply connected set in the plane, we have the following remarkable equivalence
\begin{equation}
\label{equivalence}
\lambda_1(\Omega)>0 \qquad \Longleftrightarrow \qquad r_\Omega<+\infty.
\end{equation}
Indeed, if the inradius is finite, we immediately get from \eqref{MH} that $\lambda_1(\Omega)$ must be positive. The converse implication is simpler and just based on the easy (though sharp) inequality
\[
\lambda_1(\Omega)\le \frac{\lambda_1(B_1)}{r_\Omega^2}.
\]
Here $B_1\subset\mathbb{R}^2$ is any disk with radius $1$ and the estimate simply follows from the monotonicity with respect to set inclusion of $\lambda_1$, together with its scaling properties.
\vskip.2cm\noindent
The proof in \cite{Ma} runs very similarly to that of the {\it Faber-Krahn inequality}, based on symmetrization techniques (see \cite[Chapter 3]{Hen}). It starts by rewriting the Dirichlet integral and the $L^2$ norm by using the Coarea Formula; then a clever use is made of a particular quantitative isoperimetric inequality in $\mathbb{R}^2$ (a {\it Bonnesen--type inequality}), in order to obtain a lower bound in terms of $r_\Omega$ only.
\par
It should be noticed that Makai's result has been overlooked or neglected for some years and then rediscovered independently by Hayman, by means of a completely different proof, see \cite[Theorem 1]{Ha}. For this reason, we will call \eqref{MH} the {\it Makai-Hayman inequality}.
\par
It is interesting to remark that the result by Makai is quantitatively better than the one by Hayman: indeed, the former obtains \eqref{MH} with $C=1/4$, while the latter is only able to get the poorer constant $C=1/900$ by his method of proof. 
\par
This could suggest that the attribution of this result to both authors is maybe too generous. On the contrary, we will show in this paper that, in despite of providing a poorer constant, the method of proof by Hayman is elementary, flexible and robust enough to be generalized to other situations, where Makai's and other approaches become too complicate or do not seem feasible. 
\par
In any case, we point out that the exact determination of the sharp constant in \eqref{MH}, i.\,e. 
\[
C_{MH}:=\inf\Big\{\lambda_1(\Omega)\,r_\Omega^2\, :\, \Omega\subset\mathbb{R}^2 \mbox{ simply connected with } r_\Omega<+\infty\Big\},
\] 
is still a challenging open problem.
The best result at present is that
\[
0.6197<C_{MH}<2.13,
\]
obtained by Ba\~{n}uelos and Carroll (see \cite[Corollary 1]{BC} for the lower bound and \cite[Theorem 2]{BC} for the upper bound). The upper bound has then been slightly improved by Brown in \cite{Br}, by using a refinement of the method by Ba\~nuelos and Carroll.
\par
The inequality \eqref{MH} has also been obtained by Ancona in \cite{An}, by using yet another proof. His result comes with the constant $C=1/16$, much better than Hayman's one, but still worse than that obtained by Makai. The proof by Ancona is quite elegant: it is based on the use of conformal mappings and the so-called {\it Koebe's one quarter Theorem} (see \cite[Chapter 12]{Kr}), which permits to obtain the following Hardy inequality for a simply connected set in the plane
\[
\frac{1}{16}\,\int_\Omega \frac{|\varphi|^2}{\mathrm{dist}(x,\partial\Omega)^2}\,dx\le \int_\Omega |\nabla \varphi|^2\,dx,\qquad \mbox{ for every } \varphi\in C^\infty_0(\Omega).
\]
From this, inequality \eqref{MH} is easily obtained (with $C=1/16$), by observing that 
\[
r_\Omega=\sup_{x\in\Omega} \mathrm{dist}(x,\partial\Omega),
\]
and then using the definition of $\lambda_1(\Omega)$. 
The conformality of the Dirichlet integral obviously plays a central role in this proof.
The result by Ancona is quite remarkable, as the Hardy inequality is proved without any regularity assumption on $\partial\Omega$.
A generalization of this result can be found in \cite{LS}.

\subsection{Goal of the paper and main results}
Our work is aimed at investigating the validity of a result analogous to \eqref{MH} for fractional Sobolev spaces. In order to be more precise, we need some definitions at first. Let $0<s<1$ and let us recall the definition of Gagliardo-Slobodecki\u{\i} seminorm
\[
[u]_{W^{s,2}(\mathbb{R}^N)}=\left(\iint_{\mathbb{R}^N\times\mathbb{R}^N}\frac{|u(x)-u(y)|^2}{|x-y|^{N+2\,s}}\,dx\,dy\right)^\frac{1}{2}.
\]
Accordingly, we consider the fractional Sobolev space
\[
W^{s,2}(\mathbb{R}^N)=\Big\{u\in L^2(\mathbb{R}^N)\, :\, [u]_{W^{s,2}(\mathbb{R}^N)}<+\infty\Big\},
\]
endowed with the norm
\[
\|u\|_{W^{s,2}(\mathbb{R}^N)}=\|u\|_{L^2(\mathbb{R}^N)}+[u]_{W^{s,2}(\mathbb{R}^N)}.
\]
Finally, we consider the space $\widetilde W^{s,2}_0(\Omega)$, defined as the closure of $C^\infty_0(\Omega)$ in $W^{s,2}(\mathbb{R}^N)$. Observe that by definition the elements of $\widetilde W^{s,2}_0(\Omega)$ have to be considered on the whole $\mathbb{R}^N$ and they come with a natural nonlocal homogeneous Dirichlet condition ``at infinity'', i.e. they identically vanish on the complement $\mathbb{R}^N\setminus\Omega$.
\par
We then consider the quantity
\begin{equation}\label{def l1s}
\lambda^s_1(\Omega):=\inf_{u\in \widetilde W^{s,2}_0(\Omega)\setminus\{0\}}\dfrac{[u]^2_{W^{s,2}(\mathbb{R}^N)}}{\|u\|^2_{L^2(\Omega)}}.
\end{equation}
Again by definition, this quantity is non-zero if and only if the open set $\Omega$ supports the fractional Poincar\'e inequality
\[
c\,\int_\Omega |u|^2\,dx\le \iint_{\mathbb{R}^N\times\mathbb{R}^N} \frac{|u(x)-u(y)|^2}{|x-y|^{N+2\,s}}\,dx,\qquad \mbox{ for every } u\in C^\infty_0(\Omega).
\]
As in the local case, whenever the infimum in \eqref{def l1s} is attained, this quantity will be called {\it first eigenvalue of the fractional Dirichlet-Laplacian of order $s$}. We recall that the latter is the linear operator denoted by the symbol $(-\Delta)^s$ and defined in weak form by
\[
\langle (-\Delta)^s u,\varphi\rangle =\iint_{\mathbb{R}^N\times\mathbb{R}^N} \frac{(u(x)-u(y))\,(\varphi(x)-\varphi(y))}{|x-y|^{N+2\,s}}\,dx\,dy,\qquad \mbox{ for every } \varphi\in C^\infty_0(\Omega).
\]
In this paper we want to inquire to which extent the Makai-Hayman inequality \eqref{MH} still holds for $\lambda_1^s$ defined above, still in the case of simply connected sets in the plane. Our main results assert that such an inequality is possible, provided $s$ is ``large enough''. More precisely, we have the following
\begin{teo}[Fractional Makai-Hayman inequality]
\label{teo:teorema principale}
Let $1/2<s<1$ and let $\Omega\subset\mathbb{R}^2$ be an open simply connected set, with finite inradius $r_\Omega$. There exists an explicit universal constant $\mathcal{C}_s>0$ such that
\begin{equation}
\label{MHs}
\lambda_1^s(\Omega)\ge \frac{\mathcal{C}_s}{r_\Omega^{2\,s}}.
\end{equation}
Moreover, $\mathcal{C}_s$ has the following asymptotic behaviours
\[
\mathcal{C}_s\sim \left(s-\frac{1}{2}\right),\ \mbox{ for } s\searrow \frac{1}{2}\qquad \mbox{ and }\qquad 
\mathcal{C}_s\sim \frac{1}{1-s},\ \mbox{ for } s\nearrow 1.
\]
\end{teo}
\begin{oss}
We point out that the constant $\mathcal{C}_s$ appearing in the above estimate exhibits the sharp asymptotic dependence on $s$, as $s\nearrow1$. Indeed, by recalling that for every open set $\Omega\subset \mathbb{R}^N$ we have (see \cite[Lemma A.1]{BCV})
\begin{equation}
\label{BBM}
\limsup_{s\nearrow 1} (1-s)\,\lambda_1^s(\Omega)\le C_N\,\lambda_1(\Omega),	
\end{equation}
from Theorem \ref{teo:teorema principale} we can obtain the usual Makai-Hayman inequality for the Dirichlet-Laplacian, possibly with a worse constant. We recall that \eqref{BBM} is based on the fundamental asymptotic result by Bourgain, Brezis and Mironescu for the Gagliardo-Slobodecki\u{\i} seminorm, see \cite{BBM}. We refer to \cite{BSY} for some interesting refinements of such a result.
\end{oss}

The previous result is complemented by the next one, asserting that for $0<s\le1/2$ a fractional Makai-Hayman inequality {\it is not} possible. In this way, we see that even for $s\searrow 1/2$ the asymptotic behaviour of $\mathcal{C}_s$ is optimal, in a sense.
\begin{teo}[Counter-example for $0<s\le 1/2$]
\label{teo:teor valori ammissibili}
There exists a sequence $\{Q_n\}_{n\in\mathbb{N}}\subset\mathbb{R}^2$ of open bounded simply connected sets such that 
\[
0<r_{Q_n}\le C,\qquad \mbox{ for every } n\in\mathbb{N},
\]
and  
\[
\lim_{n\to\infty}\lambda_1^s(Q_n)=0,\qquad \mbox{ for every }0<s\le \frac{1}{2}.
\]
\end{teo}
\begin{oss}
In \cite[Theorem 1.1]{CR}, a different counter-example for $0<s<1/2$ is given. Apart from the fact that in \cite{CR} the borderline case $s=1/2$ is not considered, one could observe that strictly speaking the counter-example in \cite{CR} is not a simply connected set, since it is made of countably many connected components. As it will be apparent to the experienced reader, our example clearly displays the role of fractional $s-$capacity in the failure of the Makai-Hayman inequality for $0<s\le 1/2$ (see for example \cite[Chapter 10, Section 4]{Maz} for fractional capacities). Indeed, the range $0<s\le 1/2$ is precisely the one for which {\it lines have zero fractional $s-$capacity}. However, even if this is the ultimate reason for such a failure, our proof will be elementary and will not explicitly appeal to the properties of capacities.
\end{oss}

\begin{oss}
Geometric estimates for eigenvalues of $(-\Delta)^s$ aroused great interest in the last years, also in the field of stochastic processes. Indeed, it is well-known that this operator is the infinitesimal generator of a symmetric $(2\,s)-$stable L\'evy process. We recall that the nonlocal homogeneous Dirichlet boundary condition considered above (i.\,e. $u\equiv 0$ on $\mathbb{R}^N\setminus\Omega$) corresponds to a  process where particles are ``killed'' upon reaching the complement of the set $\Omega$. The Gagliardo-Slobodecki\u{\i} seminorm corresponds to the so-called {\it Dirichlet form} associated to this process. For more details, we refer for example to \cite[Section 2]{BBC} and the references therein.
\par
In this context, we wish to mention the papers \cite{BMH, BLMH} and \cite{MH}, where some geometric estimates for $\lambda_1^s$ are obtained, by exploiting this probabilistic approach. In particular, the paper \cite{BLMH} is very much related to ours, since in \cite[Corollary 1]{BLMH} it is proved the lower bound
\[
\lambda_1^s(\Omega)\ge \frac{C}{r_\Omega^{2\,s}},
\]
in the restricted class of open {\it convex} subsets of the plane, with the sharp constant $C$. This result can be seen as the fractional counterpart of a well-known result for the Laplacian, which goes under the name of {\it Hersch-Protter inequality}, see \cite{He, Pr}. 
\end{oss}

\subsection{Method of proof}

As already announced at the beginning, we will achieve the result of Theorem \ref{teo:teorema principale} by adapting to our setting Hayman's proof. It is then useful to recall the key ingredients of such a proof. These are essentially two:
\begin{itemize}
\item[1.] a covering lemma, asserting that it is possible to cover an open subset $\Omega\subset\mathbb{R}^2$ with $r_\Omega<+\infty$ by means of {\it boundary disks}, whose radius is universally comparable to $r_\Omega$ and which do not overlap ``too much''  with each other. Here by {\it boundary disk} we simply mean a disk centered at the boundary $\partial\Omega$; 
\vskip.2cm
\item[2.] a Poincar\'e inequality for boundary disks in a simply connected set.
\end{itemize}
Point 1. is purely geometrical and thus it can still be used in the fractional setting. 
\par
On the contrary, the proof of point 2. is vey much local. Indeed, an essential feature of the proof in \cite{Ha} is the fact that 
\begin{equation}
\label{tangenziale}
|\nabla u|^2\ge \frac{1}{\varrho^2}\,|\partial_\theta u|^2,
\end{equation}
where $(\varrho,\theta)$ denote the usual polar coordinates.
Then one observes that a boundary circle always meets the complement of $\Omega$, when the latter is simply connected. Thus, taken a function $u\in C^\infty_0(\Omega)$, the periodic function $\theta\mapsto u(\varrho,\theta)$ vanishes somewhere in $[0,2\,\pi]$. Consequently, it satisfies the following one-dimensional Poincar\'e inequality on the interval
 \begin{equation}
 \label{1dintro}
 \int_0^{2\,\pi} |u(\varrho,\theta)|^2\,d\theta\le C\,\int_0^{2\,\pi} |\partial_\theta u(\varrho,\theta)|^2\,d\theta. 
 \end{equation}
In a nutshell, this permits to prove point 2.\,by ``foliating'' the boundary disk with concentric boundary circles, using \eqref{1dintro} on each of these circles, then integrating with respect to the radius of the circle and finally appealing to \eqref{tangenziale}.
\par
In the fractional case, the property \eqref{tangenziale} has no counterpart, because of the nonlocality of the Gagliardo-Slobodecki\u{\i} seminorm. Consequently, adapting this method to prove a fractional Poincar\'e inequality for boundary disks is a bit involved. We will achieve this through a lengthy though elementary method, which we believe to be of independent interest.

\begin{oss}[Other proofs?]
We conclude the introduction, by observing that it does not seem easy to prove \eqref{MHs} by adapting Makai's proof, because of the lack of a genuine Coarea Formula for Gagliardo-Slobodecki\u{\i} seminorms. The proof by Ancona seems to be even more prohibitive to be adapted, because of the rigid machinery of conformal mappings on which is based. In passing, we mention that it would be interesting to know whether his Hardy inequality for simply connected sets in the plane could be extended to fractional Sobolev spaces. For completeness, we refer to \cite{EHV} for some fractional Hardy inequalities under minimal regularity assumptions.
\end{oss}

\subsection{Plan of the paper}
In Section \ref{sec:2} we set the main notations and present some technical tools, needed throughout the paper. In particular, we recall Hayman's covering lemma from \cite{Ha}
and present a couple of technical results on fractional Sobolev spaces.
\par 
In Section \ref{sec:3} we prove a Poincar\'e inequality for boundary disks. This is the main ingredient for the proof of the fractional Makai-Hayman inequality.
\par 
Section \ref{sec:4} is then devoted to the proof of Theorem \ref{teo:teorema principale}, while the construction of the counter-example of Theorem  \ref{teo:teor valori ammissibili} is contained in Section \ref{sec:5}.\par 
\par
Finally, in Section \ref{sec:6} we highlight some consequences of our main result. Among these, we record a Cheeger-type inequality, a comparison result for $\lambda_1^s$ and $\lambda_1$ and the fractional analogous of the characterization \eqref{equivalence}.
\par
The paper concludes with Appendix \ref{appendice}, containing 
a one-dimensional fractional Poincar\'e inequality for periodic functions vanishing at one point (see Proposition \ref{prop poincare angolare}). This is the cornerstone on which the result in Section \ref{sec:3} is built.

\begin{ack}
The first author is a member of the Gruppo Nazionale per l'Analisi Matematica, la Probabilit\`a
e le loro Applicazioni (GNAMPA) of the Istituto Nazionale di Alta Matematica (INdAM). Both authors have been financially supported by the Italian grant FFABR {\it Fondo Per il Finanziamento delle attivit\`a di base}.
\par
Part of this work has been written during the conference ``Variational methods and applications'', held at the {\it Centro di Ricerca Matematica ``Ennio De Giorgi''} in September 2021. The organizers and the hosting institution are gratefully acknowledged.
\end{ack}

\section{Preliminaries}
\label{sec:2}

\subsection{Notation}
Given $x_0\in\mathbb{R}^N$ and $R>0$, we will denote by $B_R(x_0)$ the $N-$dimensional open ball with radius $R$ and center $x_0$. When the center coincides with the origin, we will simply write $B_R$. We indicate by $\omega_N$ the $N-$dimensional Lebesgue measure of $B_1$, so that by scaling
\[
|B_R(x_0)|=\omega_N\,R^N.
\]
If $E\subset\mathbb{R}^N$ is a measurable set with positive measure and $u\in L^1(E)$, we will use the notation
\[
\overline{u}_E:=\fint_E u(x)\,dx=\frac{1}{|E|}\int_E u(x)\,dx.
\]
For $0<s<1$ and for a measurable set $E\subset\mathbb{R}^N$, we will indicate by
\[
W^{s,2}(E)=\Big\{u\in L^2(E)\, :\, [u]_{W^{s,2}(E)}<+\infty\Big\},
\]
where 
\[
[u]_{W^{s,2}(E)}=\left(\iint_{E\times E}\frac{|u(x)-u(y)|^2}{|x-y|^{N+2\,s}}\,dx\,dy\right)^\frac{1}{2}.
\]
This space will be endowed with the norm
\[
\|u\|_{W^{s,2}(E)}=\|u\|_{L^2(E)}+[u]_{W^{s,2}(E)}.
\]
We observe that the following {\it Leibniz--type rule} holds
\begin{equation}
\label{leibniz}
[u\,v]_{W^{s,2}(E)}\le [u]_{W^{s,2}(E)}\,\|v\|_{L^\infty(E)}+[v]_{W^{s,2}(E)}\,\|u\|_{L^\infty(E)},\ \mbox{ for every }u,v\in W^{s,2}(E)\cap L^\infty(E).
\end{equation}
This will be useful somewhere in the paper.
\par
Finally, by $W^{s,2}_{\rm loc}(\mathbb{R}^N)$ we mean the collection of functions which are in $W^{s,2}(B_R)$, for every $R>0$.
\subsection{Technical tools}
In order to prove Theorem \ref{teo:teorema principale}, we will need the following covering Lemma, whose proof can be found in \cite[Lemma 2]{Ha}. The result in \cite{Ha} is stated for bounded sets and, accordingly, the relevant covering is made of a {\it finite} number of balls. However, a closer inspection of the proof in \cite{Ha} easily shows that the same result still holds by removing the boundedness assumption. In this case, the covering could be made of countably infinitely many balls: this is still enough for our purposes. We omit the proof, since it is exactly the same as in \cite{Ha}.
\begin{lemma}\label{hayman2}
	Let $\Omega\subset\mathbb{R}^2$ be an open set, with finite inradius $r_\Omega$. Then there exist at most countably many distinct points $\{z_{n}\}_{n\in\mathbb{N}}\subset\partial\Omega$ such that the family of disks 
	\[
	\mathfrak{B}=\big\{B_r(z_n)\big\}_{n\in\mathbb{N}},\qquad\mbox{with }r=r_\Omega\,\left(1+\sqrt{2}\right),
	\]
	is a covering of $\Omega$. Moreover, $\mathfrak{B}$ can be split in at most $36$ subfamilies $\mathfrak{B}_1,\ldots,\mathfrak{B}_{36}$ such that
	\[
	B_r(z_n)\cap B_r(z_m)=\emptyset\qquad \mbox{ if } B_r(z_{n}),B_r(z_{m})\in \mathfrak{B}_k, \mbox{ with }m\not=n,
	\]
	for every $k=1,\ldots,36$. 
\end{lemma}
In the following technical result, we explicitly construct a continuous extension operator for fractional Sobolev spaces defined on a ball. The result is certainly well-known (see for example \cite[Theorem 5.4]{DPV}), but here we pay particular attention to the constant appearing in the continuity estimate \eqref{continuo} below: indeed, this can be taken to be independent of the differentiability index $s$. 
\begin{lemma}
\label{lm:kelvin}
	Let $0<s<1$, there exists a linear extension operator 
	\[
	\mathcal{E}:W^{s,2}(B_1(x_0))\to W^{s,2}_{\rm loc}(\mathbb{R}^N),
	\] 
	such that for every $u\in W^{s,2}(B_1(x_0))$ and every $R>1$ we have  
	\begin{equation}
	\label{continuo}
	\big[\mathcal{E}(u)\big]_{W^{s,2}(B_R(x_0))}\le 4\,R^{4\,N}\,[u]_{W^{s,2}(B_1(x_0))},\qquad\|\mathcal{E}(u)\|_{L^2(B_R(x_0))}\le 2\,R^{2\,N}\,\|u\|_{L^2(B_1(x_0))}.
	\end{equation}
\end{lemma}
\begin{proof}
Without loss of generality, we can suppose that $x_0$ coincides with the origin.
Then, let us recall the definition of {\it inversion with respect to $\mathbb{S}^{N-1}$}: this is the bijection $\mathcal{K}:\mathbb{R}^N\setminus\{0\}\to \mathbb{R}^N\setminus\{0\}$, given by
	\[
	\mathcal{K}(x)=\frac{x}{|x|^2},\qquad \mbox{ for every }x\in\mathbb{R}^N\setminus\{0\}.
	\]
	It is easily seen that if $x\in B_R\setminus B_1$, then $\mathcal{K}(x)\in B_1\setminus B_{1/R}$. Moreover, we have  
\[
\mathcal{K}^{-1}(x)=\mathcal{K}(x)\quad \mbox{ and }\quad |\mathrm{det}(D\mathcal{K}(x))|=\frac{1}{|x|^{2\,N}},\qquad \mbox{ for every }x\in\mathbb{R}^N\setminus \{0\}.
\]
For every $u\in W^{s,2}(B_1)$, we define the extended function $	\mathcal{E}[u]$ given by
	\begin{equation}\label{def trasformata di kelvin u}
	\mathcal{E}[u](x)=\begin{cases}
	u(x),\qquad&\mbox{if }x\in B_1,\\&\\
	u(\mathcal{K}(x))&\mbox{if }x\in \mathbb{R}^N\setminus B_1.
	\end{cases}
	\end{equation}
It is easily seen that the operator $u\mapsto \mathcal{E}[u]$ is linear. In order to prove that $\mathcal{E}[u]\in W^{s,2}_{\rm loc}(\mathbb{R}^N)$, together with the claimed estimate \eqref{continuo}, we take $R>1$ and we split the seminorm of $\mathcal{E}[u]$ as follows
	\[
	\begin{split}
	\big[\mathcal{E}(u)\big]_{W^{s,2}(B_R)}^2&
	=[u]^2_{W^{s,2}(B_1)} \\&+\iint_{(B_R\setminus B_1)\times(B_R\setminus B_1)}\frac{|{u}(\mathcal{K}(x))-{u}(\mathcal{K}(y))|^2}{|x-y|^{N+2\,s}}\,dx\,dy
	\\&
	+2\,\iint_{B_1\times (B_R\setminus B_1)}\frac{|{u}(x)-{u}(\mathcal{K}(y))|^2}{|x-y|^{N+2\,s}}\,dx\,dy.
	\\
	\end{split}
	\]
By performing the change of variable $z=\mathcal{K}(x)$ in the second term on the right-hand side and the change of variable $w=\mathcal{K}(y)$ in the	second and third terms, we get
\[
\begin{split}
	\big[\mathcal{E}(u)\big]_{W^{s,2}(B_R)}^2&
	=[u]^2_{W^{s,2}(B_1)} \\&+\iint_{(B_1\setminus B_\frac{1}{R})\times(B_1\setminus B_\frac{1}{R})}\frac{|u(z)-u(w)|^2}{|\mathcal{K}^{-1}(z)-\mathcal{K}^{-1}(w)|^{N+2\,s}}\,|\det D \mathcal{K}^{-1}(z)|\,|\det D \mathcal{K}^{-1}(w)|\,dz\,dw
	\\&
	+2\,\int_{B_1\times (B_1\setminus B_\frac{1}{R})}\frac{|u(x)-u(w)|^2}{|x-\mathcal{K}^{-1}(w)|^{N+2\,s}}\,|\det D \mathcal{K}^{-1}(w)|\,dx\,dw.
	\end{split}
\]
By using the expression for the Jacobian determinant, we then easily get
\begin{equation}
\label{hey}
\begin{split}
	\big[\mathcal{E}(u)\big]_{W^{s,2}(B_R)}^2&
	\le[u]^2_{W^{s,2}(B_1)} \\
	&+R^{4\,N}\,\iint_{(B_1\setminus B_\frac{1}{R})\times (B_1\setminus B_\frac{1}{R})} \frac{|u(z)-u(w)|^2}{|\mathcal{K}^{-1}(z)-\mathcal{K}^{-1}(w)|^{N+2\,s}}\,dz\,dw\\&
	+2\,R^{2\,N}\,\iint_{B_1\times (B_1\setminus B_\frac{1}{R})}\frac{|u(x)-u(w)|^2}{|x-\mathcal{K}^{-1}(w)|^{N+2\,s}}\,dx\,dw.
	\end{split}
\end{equation}
In order to estimate the last two integrals, it is sufficient to use that
	\begin{equation}\label{I2}
	|\mathcal{K}^{-1}(z)-\mathcal{K}^{-1}(w)|=\left|\frac{1}{|z|^2}\,z-\frac{1}{|w|^2}\,w\right|\ge|z-w|,\qquad \mbox{ for every }z,w\in B_1\setminus\{0\},
	\end{equation}
	and
	\begin{equation}\label{I3}
	|x-\mathcal{K}^{-1}(w)|=\left|x-\frac{1}{|w|^2}\,w\right|\ge|{x}-w|,\qquad \mbox{ for every } x,w\in B_1\setminus\{0\}.
	\end{equation}
Indeed, by taking the square, we see that \eqref{I2} is equivalent to
	\[
	\left(\frac{1}{|z|^2}-|z|^2\right)+\left(\frac{1}{|w|^2}-|w|^2\right)\ge 2\,\left(\frac{1}{|z|^2\,|w|^2}-1\right)\,\langle z,w\rangle.
	\]
This in turn follows from Young's inequality
\[
2\,\langle z,w\rangle\le |z|^2+|w|^2,
\]
once we multiply both sides by the positive quantity 
\[
\left(\frac{1}{|z|^2\,|w|^2}-1\right).
\]
As for inequality \eqref{I3}, by taking again the square we see that the latter is equivalent to 
\begin{equation}\label{I33}
	\frac{1}{|w|^2}-|w|^2\ge2\,\left(\frac{1}{|w|^2}-1\right)\,\langle x,w\rangle.
	\end{equation}
This in turn follows again from Young's inequality: more precisely, by using that $|x|<1$, we have
\[
2\,\langle x,w\rangle\le |x|^2+|w|^2\le 1+|w|^2,
\]	
and if we now multiply both sides by the positive quantity (here we use that $|w|<1$)
\[
\left(\frac{1}{|w|^2}-1\right),
\]
we get \eqref{I33}, with some simple algebraic manipulations. 
\par
By applying the estimates \eqref{I2} and \eqref{I3} in \eqref{hey}, we finally get
\[
\begin{split}
	\big[\mathcal{E}(u)\big]_{W^{s,2}(B_R)}^2&
	\le[u]^2_{W^{s,2}(B_1)} \\
	&+R^{4\,N}\,\iint_{(B_1\setminus B_\frac{1}{R})\times (B_1\setminus B_\frac{1}{R})}\frac{|u(z)-u(w)|^2}{|z-w|^{N+2\,s}}\,dz\,dw\\&
	+2\,R^{2\,N}\,\iint_{B_1\times (B_1\setminus B_\frac{1}{R})}\frac{|u(x)-u(w)|^2}{|x-w|^{N+2\,s}}\,dx\,dw\\
	&\le \left(1+R^{4\,N}+2\,R^{2\,N}\right)\,[u]^2_{W^{s,2}(B_1)},
	\end{split}
\]
which proves the first estimate in \eqref{continuo}.
\par
We are left with estimating the $L^2$ norm of $\mathcal{E}[u]$. This is simpler and can be done as follows
	\[\begin{split}
	\int_{B_R}|\mathcal{E}(u(x))|^2\,dx&=\int_{B_1}|u(x)|^2\,dx+\int_{B_R\setminus B_1}|u(\mathcal{K}(x))|^2\,dx\\&\le
	\int_{B_1}|u(x)|^2\,dx+R^{2\,N}\,\int_{B_1\setminus B_\frac{1}{R}}|u(z)|^2\,dz\le \left(1+R^{2\,N}\right)\,\int_{B_1}|u(x)|^2\,dx.
	\end{split}
	\]
	This concludes the proof.
\end{proof}
\begin{oss}
Another important feature of the previous result is that, rather than the usual continuity estimate 
\[
\big\|\mathcal{E}[u]\big\|_{W^{s,2}(B_R)}\le C\,\|u\|_{W^{s,2}(B_1)},
\]
for the extension operator, we obtained the more precise estimate \eqref{continuo}. This will be useful in the next result.
\end{oss}

\begin{prop}\label{prop mazya shaposnikova}
	Let $0<s<1$ and let $E\subseteq B_R(x_0)\subset\mathbb{R}^N$ be a measurable set, with positive measure. 
	There exists a constant $\mathcal{M}=\mathcal{M}(N)>0$ such that for every $u\in W^{s,2}(B_R(x_0))$ 
	we have
\[
	\|u-\overline{u}_E\|^2_{L^2(B_R(x_0))}\le \mathcal{M}\,(1-s)\frac{R^{N}}{|E|}\,R^{2\,s}\,[u]^2_{W^{s,2}(B_R(x_0))}.
\]
\end{prop}
\begin{proof}
By a standard scaling argument, it is sufficient to prove the result for $R=1$ and $x_0=0$. For every $t>0$, we denote by $Q_t=(-t/2,t/2)^N$ the $N-$dimensional open cube centered at the origin, with side length $t$.
\par
We consider the extension $\mathcal{E}(u)$ of $u$ to the whole $\mathbb{R}^N$, as in \eqref{def trasformata di kelvin u}. For ease of notation, we will simply write $\widetilde{u}:=\mathcal{E}(u)$. By using the triangle inequality and the fact that $B_1\subset Q_2$, we have
\begin{equation}
\label{quattro...}
\begin{split}
\|u-\overline{u}_E\|^2_{L^2(B_1)}&\le \|\widetilde u-\overline{u}_E\|^2_{L^2(Q_2)}\\
&\le 2\,\|\widetilde u-\overline{\widetilde{u}}_{Q_2}\|^2_{L^2(Q_2)}+2\,\|\overline{\widetilde{u}}_{Q_2}-\overline{u}_E\|^2_{L^2(Q_2)}.
\end{split}
\end{equation}
By using Jensen's inequality and the fact that $|Q_2|=2^N$, we can estimate the second term as follows
\[
\begin{split}
\|\overline{\widetilde{u}}_{Q_2}-\overline{u}_E\|^2_{L^2(Q_2)}&=2^N\,|\overline{\widetilde{u}}_{Q_2}-\overline{u}_E|^2\\
&=2^N\,\left|\fint_E (\widetilde{u}(x)-\overline{\widetilde{u}}_{Q_2})\,dx \right|^2\\
&\le 2^N\,\fint_E |\widetilde{u}(x)-\overline{\widetilde{u}}_{Q_2}|^2\,dx\le \frac{2^N}{|E|}\,\|\widetilde{u}-\overline{\widetilde{u}}_{Q_2}\|_{L^2(Q_2)}^2.
\end{split}
\]
Thus from \eqref{quattro...} we get
\[
\|u-\overline{u}_E\|^2_{L^2(B_1)}\le 2\,\left(1+\frac{2^N}{|E|}\right)\,\|\widetilde{u}-\overline{\widetilde{u}}_{Q_2}\|_{L^2(Q_2)}^2.
\]
We can now apply the following fractional Poincar\'e inequality\footnote{We remark that the presence of the factor $(1-s)$ is important for our scopes. If one is not interested in keeping track of this factor, actually the proof would be much simpler, see for example \cite[page 297]{Mi2003}.} proved by Maz'ya and Shaposnikova (see \cite[page 300]{MS2})
\[
\|\varphi-\overline{\varphi}_{Q_2}\|^2_{L^2(Q_2)}\le C_N\, (1-s)\,[\varphi]_{W^{s,2}(Q_2)}^2,\qquad \mbox{ for every }\varphi\in W^{s,2}(Q_2).
\]
Here $C_N$ is an explicit dimensional constant.
This yields
\[
\begin{split}
\|u-\overline{u}_E\|^2_{L^2(B_1)}&\le 2\,\left(1+\frac{2^N}{|E|}\right)\,C_N\, (1-s)\,[\widetilde{u}]_{W^{s,2}(Q_2)}^2\\
&\le 2\,\frac{\omega_N+2^N}{|E|}\,C_N\,(1-s)\,[\widetilde{u}]_{W^{s,2}(B_{\sqrt{2}})}^2,
\end{split}
\]
where we used that $Q_2\subset B_{\sqrt{2}}$. It is now sufficient to apply Lemma \ref{lm:kelvin} with $R=\sqrt{2}$, to get the claimed conclusion. 
\end{proof}

\section{An expedient Poincar\'e inequality}
\label{sec:3}

The following result is a nonlocal counterpart of \cite[Lemma 1]{Ha} in Hayman's paper. In the proof we pay due attention to the dependence of the constant on the fractional parameter $s$, as always.

\begin{prop}[Poincar\'e for boundary disks]
\label{teor poicare per palle centrate sul bordo}
	Let $1/2<s<1$ and let $\Omega\subset\mathbb{R}^2$ be an open simply connected set, with $\partial\Omega\not=\emptyset$.
	There exists a universal constant $\mathcal{T}_s>0$ such that for every $r>0$ and every $x_0\in\partial\Omega$, we have 
	\[
	\frac{\mathcal{T}_s}{r^{2\,s}}\,\int_{B_r(x_0)}|u(x)|^2\,dx\le \iint_{B_r(x_0)\times\mathbb{R}^2}\frac{|u(x)-u(y)|^2}{|x-y|^{2+2\,s}}\,dx\,dy,\qquad \mbox{ for every } 
	u\in C^\infty_0(\Omega).
	\]
	Moreover, $\mathcal{T}_s$ has the following asymptotic behaviours
	\[
	\mathcal{T}_s\sim \left(s-\frac{1}{2}\right),\quad \mbox{ for } s\searrow \frac{1}{2}\qquad	and\qquad
	\mathcal{T}_s\sim \frac{1}{1-s},\quad \mbox{ for } s\nearrow 1.
	\]
\end{prop}
\begin{proof} Up to scaling and translating, we can assume without loss of generality that $r=1$ and that $x_0$ coincides with the origin.\par 
	We split the proof in three main steps: we first show that it is sufficient to prove the claimed estimate for the boundary ring $B_1\setminus B_{1/2}$. Then we prove such an estimate and at last we discuss the asymptotic behaviour of the constant obtained.
	\vskip.2cm\noindent
{\bf Step 1: reduction to a ring}.	Let $u\in C^\infty_0(\Omega)$, we then estimate the $L^2$ norm on $B_1$ as follows
\[
\begin{split}
\int_{B_1}|u(x)|^2\,dx&=\int_{B_1\setminus B_{1/2}}|u(x)|^2\,dx+\int_{B_{1/2}}|u(x)|^2\,dx\\
&\le \int_{B_1\setminus B_{1/2}}|u(x)|^2\,dx+2\,\int_{B_{1/2}}|u(x)-\overline{u}_{B_1\setminus B_{1/2}}|^2\,dx+2\,\int_{B_{1/2}}|\overline{u}_{B_1\setminus B_{1/2}}|^2\,dx\\
&\le \int_{B_1\setminus B_{1/2}}|u(x)|^2\,dx+2\,\int_{B_1}|u(x)-\overline{u}_{B_1\setminus B_{1/2}}|^2\,dx+2\,\int_{B_{1/2}}|\overline{u}_{B_1\setminus B_{1/2}}|^2\,dx\\
&\le \int_{B_1\setminus B_{1/2}}|u(x)|^2\,dx+2\,\int_{B_1}|u(x)-\overline{u}_{B_1\setminus B_{1/2}}|^2\,dx+\frac{2}{3}\,\int_{B_1\setminus B_{1/2}}|u(x)|^2\,dx,
\end{split}
\]
where we used the elementary inequality $(a+b)^2\le 2\,a^2+2\,b^2$ and Jensen's inequality. If we now apply Proposition \ref{prop mazya shaposnikova} with $R=1$ and $E=B_1\setminus B_{1/2}$, we get
\[
\int_{B_1}|u(x)|^2\,dx\le \frac{5}{3}\,\int_{B_1\setminus B_{1/2}}|u(x)|^2\,dx+\frac{8}{3\,\pi}\,\mathcal{M}\,(1-s)\,[u]^2_{W^{s,2}(B_1)}.
\]
Thus, in order to conclude, it is sufficient to prove that there exists a constant $C=C(s)>0$ such that 
\begin{equation}\label{claim step 1}
\int_{B_{1}\setminus B_{1/2}}|u(x)|^2\,dx\le C\,\iint_{B_1\times \mathbb{R}^2}\frac{|u(x)-u(y)|^2}{|x-y|^{2+2\,s}}\,dx\,dy,\qquad\mbox{ for every }
 u\in C^\infty_0(\Omega).
\end{equation}

	\vskip.2cm\noindent
{\bf Step 2: estimate on the ring}. We start with a topological observation.
Since we are assuming that $0\in \partial\Omega$ and that $\Omega$ is simply connected, we have the following crucial property	
\begin{equation}
\label{nullomotopo}
\partial B_\varrho\cap (\mathbb{R}^2\setminus \Omega)\not=\emptyset,\qquad \mbox{ for every }\varrho>0.
\end{equation}	
Indeed, if this were not true, we would have existence of a circle entirely contained in $\Omega$ and centered on the boundary of $\partial\Omega$. Such a circle could not be null-homotopic in $\Omega$, thus contradicting our topological assumption.
\par
In the rest of the proof, we will use polar coordinates $(\varrho,\theta)$ and we will make the slight abuse of notation of writing
$u(\varrho,\theta)$.
Then, in light of the property \eqref{nullomotopo}, for each $\varrho\in(1/2,1)$ there exists $\theta_\varrho\in[0,2\pi)$ such that $\theta\mapsto u(\varrho,\theta)$ must vanish at $\theta_\varrho$. Hence, for every $1/2<\varrho<1$ we can apply Proposition \ref{prop poincare angolare} to the function $\theta\mapsto u(\varrho,\theta)$ and get
\[
\int_0^{2\,\pi}|u(\varrho,\theta)|^2\,d\theta\le  \frac{1}{\mu_s}\,\int_0^{2\pi}\int_0^{2\pi}\frac{|u(\varrho,\theta)-u(\varrho,\varphi)|^2}{|\theta-\varphi|^{1+2\,s}_{\mathbb{S}^1}}\,d\theta\,d\varphi.
\]
The constant $\mu_s$ is the same as in Proposition \ref{prop poincare angolare} and 
\[
|\theta-\varphi|_{\mathbb{S}^1}:=\min_{k\in\mathbb{Z}} |\theta-\varphi+2\,k\,\pi|,\qquad \mbox{ for every } \theta,\varphi\in\mathbb{R}.
\]
If we now multiply both sides by $\varrho$, integrate over the interval $(1/2,1)$ and write the $L^2$ norm in polar coordinates, we get
	\begin{equation}\label{stima L2(B)}
	\begin{split}
\int_{B_1\setminus B_{1/2}}|u(x)|^2\,dx
	&\le \frac{1}{\mu_s}\,\int_\frac{1}{2}^1\int_0^{2\pi}\int_0^{2\,\pi}\frac{|u(\varrho,\theta)-u(\varrho,\varphi)|^2}{\varrho^{1+2\,s}\,{|\theta-\varphi|^{1+2\,s}_{\mathbb{S}^1}}}\,\varrho\,d\varrho\,d\theta\,d\varphi.
	\end{split}
	\end{equation}
	Observe that we further used the fact that $\varrho\le 1$, to let the term $\varrho^{-1-2\,s}$ appear. \par 
In order to achieve \eqref{claim step 1}, we need to show that the term on the right-hand side can be estimated by a two-dimensional Gagliardo-Slobodecki\u{\i} seminorm. To this aim, we follow an argument similar to that of \cite[Lemma B.2]{BB}.
At first, it is easily seen that
	\[
	\begin{split}
	\left(\varrho\,|\theta-\varphi|_{\mathbb{S}^1}\right)^{-1-2\,s}&=(1+2\,s)\,\int_0^{+\infty}\left(t+\varrho\,|\theta-\varphi|_{\mathbb{S}^1}\right)^{-2-2\,s}\,dt.
	\end{split}
	\]
	By inserting this in \eqref{stima L2(B)}, we end up with
	\begin{equation}\label{stima 2 L2(B)}
	\begin{split}
\int_{B_1\setminus B_{1/2}} &|u(x)|^2\,dx\\&\le \frac{1+2\,s}{\mu_s}\,\int_0^{2\pi}\int_0^{2\pi}\int_\frac{1}{2}^1\int_0^{+\infty}\frac{|u(\varrho,\theta)-u(\varrho,\varphi)|^2}{(t+\varrho\,|\theta-\varphi|_{\mathbb{S}^1})^{2+2\,s}}\,\varrho\,d\theta\,d\varphi\,d\varrho\,dt\\
	&\le \frac{2\,(1+2\,s)}{\mu_s}\,\int_0^{2\pi}\int_0^{2\,\pi}\int_\frac{1}{2}^1\int_0^{+\infty}\frac{|u(\varrho,\theta)-u(\varrho,\varphi)|^2}{(t+\varrho\,|\theta-\varphi|_{\mathbb{S}^1})^{2+2\,s}}\,\varrho\,(\varrho+t)\,d\theta\,d\varphi\,d\varrho\,dt,
	\end{split}
	\end{equation}
	where we have used that $1/2\le\varrho+t$. We now split the set $[0,2\,\pi]\times [0,2\,\pi]=J_{-1}\cup J_0\cup J_1$,
	where
	\[
	J_{-1}=\Big\{(\theta,\varphi)\,:\,\theta\in[0,\pi], \,\theta+\pi<\varphi\le 2\,\pi\Big\},\qquad J_{1}=\Big\{(\theta,\varphi)\,:\,\theta\in [\pi,2\,\pi],\,0\le\varphi<\theta-\pi\Big\},\]
	and
	\[
	J_0=\Big\{(\theta,\varphi)\,:\,\theta\in[0,2\,\pi],\,
	\max\{0,\theta-\pi\}\le\varphi\le\min\{2\,\pi,\theta+\pi\}\Big\},
	\]
see Figure \ref{fig:1}.	
\begin{figure}
\includegraphics[scale=.25]{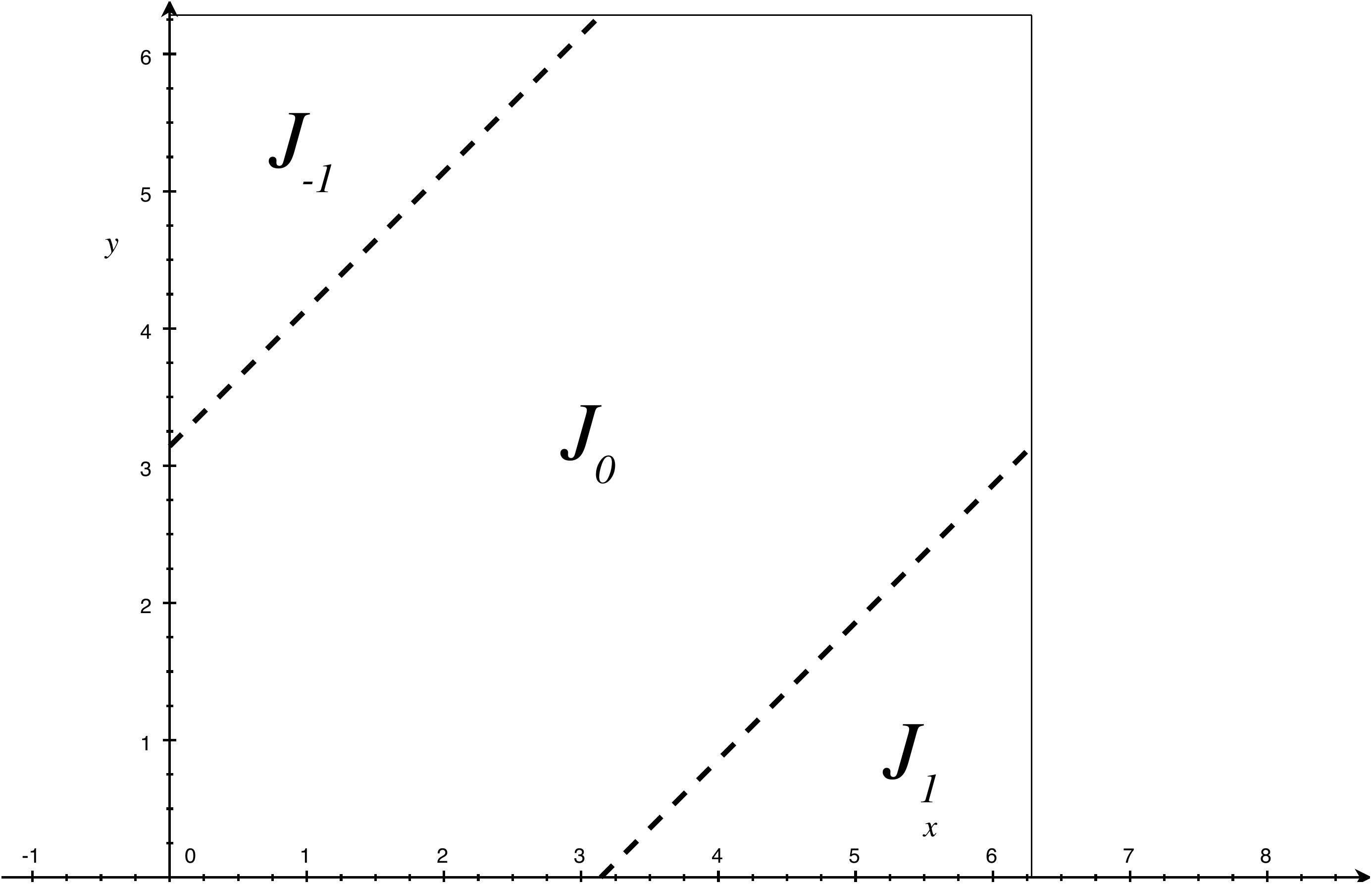}
\caption{The partition of $[0,2\,\pi]\times[0,2\,\pi]$ needed to define the midpoint function.}
\label{fig:1}
\end{figure}
	Then, we define the {\it midpoint function} by
\begin{equation}
\label{midpoint}
	\overline{\theta\,\varphi}=\frac{\theta+\varphi}{2}+\ell\,\pi,\qquad \mbox{ if } (\theta,\varphi)\in J_\ell, \mbox{ with } \ell=-1,0,1.
\end{equation}
	Thanks to the triangle inequality, we estimate the numerator in the right-hand side of \eqref{stima 2 L2(B)} as follows
	\[\begin{split}
	|u(\varrho,\theta)-u(\varrho,\varphi)|^2&\le2 \left|u\left(\varrho,\theta\right)-u\left(\varrho+t,\overline{\theta\,\varphi}\right)\right|^2+2\left|u\left(\varrho,\varphi\right)-u\left(\varrho+t,\overline{\theta\,\varphi}\right)\right|^2.
	\end{split}\]	
	As for the denominator, we observe that $|\theta-\overline{\theta\,\varphi}|_{\mathbb{S}^1}=|\varphi-\overline{\theta\,\varphi}|_{\mathbb{S}^1}$, thus we get
	\[
	|\theta-\varphi|_{\mathbb{S}^1}=2\,|\theta-\overline{\theta\,\varphi}|_{\mathbb{S}^1}\ge 2\,|e^{i\,\theta}-e^{i\,\overline{\theta\,\varphi}}|\quad \mbox{ and }\quad |\theta-\varphi|_{\mathbb{S}^1}=2\,|\varphi-\overline{\theta\,\varphi}|_{\mathbb{S}^1}\ge 2\,|e^{i\,\varphi}-e^{i\,\overline{\theta\,\varphi}}|,
	\]
	where the inequality comes from Lemma \ref{lm:equivalenza}. By using this fact, the identity $|e^{i\overline{\theta\,\varphi}}|=1$ and the triangle inequality again, we can estimate the denominator as
	\[
	\begin{split}
	t+\varrho\,|\theta-\varphi|_{\mathbb{S}^1}&\ge  t+\varrho\,|e^{i\theta}-e^{i\,\overline{\theta\,\varphi}}|\\
	&\ge \left|\varrho\,(e^{i\,\theta}-e^{i\,\overline{\theta\,\varphi}})-t\,e^{i\,\overline{\theta\,\varphi}}\right|=\left|\varrho\,e^{i\,\theta}-(\varrho+t)\,e^{i\,\overline{\theta\,\varphi}}\right|,
	\end{split}
	\]
	and similarly
	\[
	t+\varrho\,|\theta-\varphi|_{\mathbb{S}^1}\ge \left|\varrho\,e^{i\,\varphi}-(\varrho+t)\,e^{i\,\overline{\theta\,\varphi}}\right|.
	\]
	These allow us to estimate the right-hand side in \eqref{stima 2 L2(B)} in the following way:
	\[
	\begin{split}
\int_{B_1\setminus B_{1/2}} &|u(x)|^2\,dx\\
	&\le \frac{4\,(1+2\,s)}{\mu_s}\,\int_0^{2\,\pi}\int_0^{2\,\pi}\int_\frac{1}{2}^1\int_0^{+\infty}\frac{\left|u\left(\varrho,\theta\right)-u\left(\varrho+t,\overline{\theta\,\varphi}\right)\right|^2}{\left|\varrho\,e^{i\,\theta}-(\varrho+t)\,e^{i\,\overline{\theta\,\varphi}}\right|^{2+2\,s}}\,\varrho\,(\varrho+t)\,d\theta\,d\varphi\,d\varrho\,dt\\
	&+\frac{4\,(1+2\,s)}{\mu_s}\,\int_0^{2\,\pi}\int_0^{2\,\pi}\int_\frac{1}{2}^1\int_0^{+\infty}\frac{\left|u\left(\varrho,\varphi\right)-u\left(\varrho+t,\overline{\theta\,\varphi}\right)\right|^2}{\left|\varrho\,e^{i\,\varphi}-(\varrho+t)\,e^{i\,\overline{\theta\,\varphi}}\right|^{2+2\,s}}\,\varrho\,(\varrho+t)\,d\theta\,d\varphi\,d\varrho\,dt\\
	&=\frac{8\,(1+2\,s)}{\mu_s}\,\int_0^{2\,\pi}\int_0^{2\,\pi}\int_\frac{1}{2}^1\int_0^{+\infty}\frac{\left|u\left(\varrho,\theta\right)-u\left(\varrho+t,\overline{\theta\,\varphi}\right)\right|^2}{\left|\varrho\,e^{i\,\theta}-(\varrho+t)\,e^{i\,\overline{\theta\,\varphi}}\right|^{2+2\,s}}\,\varrho\,(\varrho+t)\,d\theta\,d\varphi\,d\varrho\,dt.
\end{split}
\]
In the last identity we used that both multiple integrals coincide, by symmetry of the integrands. If we now make the change of variable $\tau=\varrho+t$ 
 and use the decomposition $[0,2\,\pi]\times [0,2\,\pi]=J_{-1}\cup J_0\cup J_1$, we obtain
\begin{equation}
\label{sabato}
 \begin{split}
 \int_{B_1\setminus B_{1/2}} &|u(x)|^2\,dx\\
	&\le \frac{8\,(1+2\,s)}{\mu_s}\,\sum_{\ell=-1}^1 \iint_{J_\ell}\int_\frac{1}{2}^1\int_\varrho^{+\infty}\frac{\left|u\left(\varrho,\theta\right)-u(\tau,\overline{\theta\,\varphi})\right|^2}{\left|\varrho\,e^{i\,\theta}-\tau\,e^{i\,\overline{\theta\,\varphi}}\right|^{2+2\,s}}\,\varrho\,\tau\,d\theta\,d\varphi\,d\varrho\,d\tau.
 \end{split}
\end{equation}
If we now denote
	\[
	\widetilde{J}_{-1}=\Big\{(\theta,\varphi)\,:\,\theta\in[0,\pi],\, \theta-\pi<\varphi\le 0\Big\}\quad\mbox{and}\quad \widetilde{J}_1=\Big\{(\theta,\varphi)\,:\,\theta\in[\pi,2\,\pi],\,2\,\pi\le\varphi<\theta+\pi\Big\},
\]
use the definition of midpoint function \eqref{midpoint} and make the change of variables
\[
\begin{array}{ccc}
J_{-1} & \to &\widetilde J_{-1}\\
(\theta,\varphi) & \mapsto & (\theta,\varphi-2\,\pi)
\end{array}\qquad \mbox{ and }\qquad \begin{array}{ccc}
J_{1} & \to &\widetilde J_{1}\\
(\theta,\varphi) & \mapsto & (\theta,\varphi+2\,\pi),
\end{array}
\]
we obtain from \eqref{sabato}
\[
\begin{split}	
\int_{B_1\setminus B_{1/2}} |u(x)|^2\,dx\le \frac{8\,(1+2\,s)}{\mu_s}\,\iint_{\widetilde{J}_{-1}\cup J_0\cup \widetilde{J}_1}\int_\frac{1}{2}^1\int_\varrho^{+\infty}\frac{\left|u\left(\varrho,\theta\right)-u\left(\tau,\frac{\theta+\varphi}{2}\right)\right|^2}{\left|\varrho\,e^{i\,\theta}-\tau\,e^{i\,\frac{\theta+\varphi}{2}}\right|^{2+2s}}\,\varrho\,\tau\,d\theta\,d\varphi\,d\varrho\,d\tau.
	\end{split}
	\] 
For every $\theta\in[0,2\,\pi]$, we now make the change of variable $\gamma=(\theta+\varphi)/2$, thus the above estimate becomes 
\begin{equation}
\label{tosse}
\int_{B_1\setminus B_{1/2}} |u(x)|^2\,dx
	\le \frac{4\,(1+2\,s)}{\mu_s}\,\sum_{\ell=-1}^1\iint_{\widehat{J}_\ell}\int_\frac{1}{2}^1\int_\varrho^{+\infty}\frac{\left|u\left(\varrho,\theta\right)-u\left(\tau,\gamma\right)\right|^2}{\left|\varrho\,e^{i\,\theta}-\tau\,e^{i\,\gamma}\right|^{2+2s}}\,\varrho\,\tau\,d\theta\,d\gamma\,d\varrho\,d\tau,
\end{equation} 
where
	\[
	\widehat{J}_{-1}=\left\{(\theta,\gamma)\,:\theta\in\left[0,\frac{\pi}{2}\right],\,\theta-\frac{\pi}{2}<\gamma\le 0\right\},\quad \widehat{J}_1=\left\{(\theta,\gamma)\,:\,\theta\in\left[\frac{3}{2}\,\pi,2\,\pi\right],\,2\,\pi\le\gamma<\theta+\frac{\pi}{2}\right\},\]
	and
	\[\widehat{J}_0=\left\{(\theta,\gamma)\,:\,\theta\in[0,2\,\pi],\,
	\max\left\{0,\,\theta-\frac{\pi}{2}\right\}\le\gamma\le\min\left\{2\pi,\,\theta+\frac{\pi}{2}\right\}\right\} .
	\]
If we now exploit the $2\,\pi-$periodicity of the integrand, we have
\begin{equation}
\label{oh}
\begin{split}
\iint_{\widehat{J}_{-1}}&\int_\frac{1}{2}^1\int_\varrho^{+\infty}\frac{\left|u\left(\varrho,\theta\right)-u\left(\tau,\gamma\right)\right|^2}{\left|\varrho\,e^{i\,\theta}-\tau\,e^{i\,\gamma}\right|^{2+2s}}\,\varrho\,\tau\,d\theta\,d\gamma\,d\varrho\,d\tau\\
&=\iint_{\widehat{J}_{-1}}\int_\frac{1}{2}^1\int_\varrho^{+\infty}\frac{\left|u\left(\varrho,\theta\right)-u\left(\tau,\gamma+2\,\pi\right)\right|^2}{\left|\varrho\,e^{i\,\theta}-\tau\,e^{i\,(\gamma+2\,\pi)}\right|^{2+2s}}\,\varrho\,\tau\,d\theta\,d\gamma\,d\varrho\,d\tau\\
&=\iint_{I_{-1}}\int_\frac{1}{2}^1\int_\varrho^{+\infty}\frac{\left|u\left(\varrho,\theta\right)-u\left(\tau,\varphi\right)\right|^2}{\left|\varrho\,e^{i\theta}-\tau\,e^{i\,\varphi}\right|^{2+2s}}\,\varrho\,\tau\,d\theta\,d\varphi\,d\varrho\,d\tau,
\end{split}
\end{equation}
where we set $\varphi=\gamma+2\,\pi$ and
\[
I_{-1}=\left\{(\theta,\varphi)\,:\theta\in\left[0,\frac{\pi}{2}\right],\,\theta+\frac{3}{2}\,\pi<\varphi\le 2\,\pi\right\}.
\]
Similarly, we can obtain 
\begin{equation}
\label{oh2}
\begin{split}
\iint_{\widehat{J}_{1}}&\int_\frac{1}{2}^1\int_\varrho^{+\infty}\frac{\left|u\left(\varrho,\theta\right)-u\left(\tau,\gamma\right)\right|^2}{\left|\varrho\,e^{i\,\theta}-\tau\,e^{i\,\gamma}\right|^{2+2\,s}}\,\varrho\,\tau\,d\theta\,d\gamma\,d\varrho\,d\tau\\
&=\iint_{I_1}\int_\frac{1}{2}^1\int_\varrho^{+\infty}\frac{\left|u\left(\varrho,\theta\right)-u\left(\tau,\varphi\right)\right|^2}{\left|\varrho\,e^{i\,\theta}-\tau\,e^{i\,\varphi}\right|^{2+2s}}\,\varrho\,\tau\,d\theta\,d\varphi\,d\varrho\,d\tau,
\end{split}
\end{equation}
with the change of variable $\varphi=\gamma-2\,\pi$ and 
\[
I_1=\left\{(\theta,\varphi)\,:\,\theta\in\left[\frac{3}{2}\,\pi,2\,\pi\right],\,0\le\varphi<\theta-\frac{3}{2}\,\pi\right\}.
\]
By observing that $I_{-1}\cup \widehat{J}_0\cup I_1\subset [0,2\,\pi]\times [0,2\,\pi]$ and that the three sets $I_{-1},\widehat{J}_0$ and $I_1$ are pairwise disjoint, from \eqref{tosse}, \eqref{oh} and \eqref{oh2} we finally obtain 
\[
\begin{split}
\int_{B_1\setminus B_{1/2}} |u(x)|^2\,dx
	&\le \frac{4\,(1+2\,s)}{\mu_s}\,\iint_{[0,2\,\pi]\times[0,2\,\pi]}\int_\frac{1}{2}^1\int_\varrho^{+\infty}\frac{\left|u\left(\varrho,\theta\right)-u\left(\tau,\varphi\right)\right|^2}{\left|\varrho\,e^{i\,\theta}-\tau\,e^{i\,\varphi}\right|^{2+2\,s}}\,\varrho\,\tau\,d\theta\,d\varphi\,d\varrho\,d\tau\\
	&\le \frac{4\,(1+2\,s)}{\mu_s}\,\iint_{B_1\times \mathbb{R}^2} \frac{|u(x)-u(y)|^2}{|x-y|^{2+2\,s}}\,dx\,dy.
	\end{split}
\]
This concludes the proof of \eqref{claim step 1}.
\vskip.2cm\noindent	
{\bf Step 3: asymptotics for the constant.} From {\bf Step 1} and {\bf Step 2}, we obtained the Poincar\'e inequality claimed in the statement, with constant
\[
\mathcal{T}_s=\left(\frac{20\,(1+2\,s)}{3\,\mu_s}+\frac{8}{3\,\pi}\,\mathcal{M}\,(1-s)\right)^{-1}.
\]
By using the asymptotics for the constant $\mu_s$ (see Proposition \ref{prop poincare angolare} below), we get the desired conclusion.
\end{proof} 
\begin{oss}
The previous result can not hold for $0<s\le 1/2$. Indeed, if the result were true for $0<s\le 1/2$, this would permit to extend the fractional Makai-Hayman inequality to this range, as well (see the next section). However, this would contradict Theorem \ref{teo:teor valori ammissibili}.
\end{oss}

\section{Proof of Theorem \ref{teo:teorema principale}}
\label{sec:4}
	Without loss of generality, we can consider $r_\Omega=1$. We take $\mathfrak{B}$ and $\mathfrak{B}_1,\ldots,\mathfrak{B}_{36}$ to be respectively the covering of $\Omega$ and the subclasses given by Lemma \ref{hayman2}, made of ball with radius $r=1+\sqrt{2}$. 
	\par 
	We take an index $k\in\{1,\ldots,36\}$, then we know that $\mathfrak{B}_k$ is composed of (possibly) countably many disjoint balls with radius $r$, centered on $\partial\Omega$. We indicate by $B^{j,k}$ each of these balls.
	\par 
Then, for every $u\in C^\infty_0(\Omega)\setminus\{0\}$ we have
	\begin{equation}\label{somma}
	\iint_{\mathbb{R}^2\times\mathbb{R}^2}\frac{|u(x)-u(y)|^2}{|x-y|^{2+2\,s}}\,dx\,dy\ge\sum_{B^{j,k}\in\mathfrak{B}_k}\iint_{B^{j,k}\times \mathbb{R}^2}\frac{|u(x)-u(y)|^2}{|x-y|^{2+2\,s}}\,dx\,dy.
	\end{equation}
For each ball $B^{j,k}$, we can apply Proposition \ref{teor poicare per palle centrate sul bordo} so to obtain that
	\[
\sum_{B^{j,k}\in\mathfrak{B}_k}\iint_{B^{j,k}\times \mathbb{R}^2}\frac{|u(x)-u(y)|^2}{|x-y|^{2+2\,s}}\,dx\,dy\ge \frac{\mathcal{T}_s}{(1+\sqrt{2})^{2\,s}}\,\sum_{B^{j,k}\in\mathfrak{B}_k}\int_{B^{j,k}}|u(x)|^2\,dx.
	\]
	We insert this estimate in \eqref{somma} and then sum over $k=1,\dots,36$. We get
	\[
	\begin{split}
36\,\iint_{\mathbb{R}^2\times\mathbb{R}^2}\frac{|u(x)-u(y)|^2}{|x-y|^{2+2\,s}}\,dx\,dy&\ge \sum_{k=1}^{36}\sum_{B^{j,k}\in\mathfrak{B}_k}\iint_{B^{j,k}\times\mathbb{R}^2}\frac{|u(x)-u(y)|^2}{|x-y|^{2+2\,s}}\,dx\,dy\\
&\ge \frac{\mathcal{T}_s}{(1+\sqrt{2})^{2\,s}}\,\sum_{k=1}^{36}\sum_{B^{j,k}\in\mathfrak{B}_k}\int_{B^{j,k}}|u(x)|^2\,dx\\
&\ge \frac{\mathcal{T}_s}{(1+\sqrt{2})^{2\,s}}	\int_\Omega|u(x)|^2\,dx.
	\end{split}
	\]
In the last inequality we used that $\mathfrak{B}$ is a covering of $\Omega$.	
By recalling the definition of $\lambda_1^s(\Omega)$, from the previous chain of inequalities we thus get the claimed estimate \eqref{MH}, with constant
\[
\mathcal{C}_s:=\frac{\mathcal{T}_s}{36\,(1+\sqrt{2})^{2\,s}}.
\]
The asymptotic behaviour of $\mathcal{C}_s$ can now be inferred from that of $\mathcal{T}_s$, which in turn is contained in Proposition \ref{teor poicare per palle centrate sul bordo}.

\begin{oss}
For suitable classes of open sets in $\mathbb{R}^N$ and every $0<s<1$, it is possible to give a Makai-Hayman--type lower bound on $\lambda_1^s$, by taking advantage of the nonlocality of the Gagliardo-Slobodecki\u{\i} seminorm. More precisely, this is possible provided $\Omega$ satisfies the following mild regularity assumption: there exist\footnote{It is not difficult to see that this property never holds for $\sigma=1$.} $\sigma>1$ and $\alpha>0$ 
\begin{equation}
\label{francesca}
\frac{|B_{\sigma\, r_\Omega}(x)\setminus\Omega|}{|B_{\sigma\,r_\Omega}(x)|}\ge \alpha,\qquad \mbox{ for every } x\in\Omega.
\end{equation}
	Indeed, in this case for every $u\in C^\infty_0(\Omega)$ we can simply estimate 
	\[
	\begin{split}
	\iint_{\mathbb{R}^N\times\mathbb{R}^N}\frac{|u(x)-u(y)|^2}{|x-y|^{N+2\,s}}\,dx\,dy&\ge\int_{\mathbb{R}^N}\left(\int_{B_{\sigma\, r_\Omega}(x)\setminus \Omega}\frac{|u(x)|^2}{|x-y|^{N+2\,s}}\,dy\right)\,dx\\
	&\ge\frac{1}{(\sigma\, r_\Omega)^{N+2\,s}}\,\int_{\mathbb{R}^N}|B_{\sigma\, r_\Omega}(x)\setminus\Omega|\,|u(x)|^2\,dx\\
	&\ge\frac{\alpha\,\omega_N}{{(\sigma\, r_\Omega)}^{2\,s}}\,\int_{\Omega}|u(x)|^2\,dx,
	\end{split}
	\]
where in the last inequality we used the additional condition. By arbitrariness of $u$, we get
\[
\lambda_1^s(\Omega)\ge 	\frac{\alpha\,\omega_N}{\sigma^{2\,s}}\,\frac{1}{r_\Omega^{2\,s}}.
\]
One could observe that the additional condition \eqref{francesca} does not always hold for a simply connected set in the plane. Moreover, the constant obtained in this way is quite poor: first of all, it is not universal. It depends on the parameters $\alpha$ and $\sigma$ and it deteriorates as $\sigma\searrow 1$, since in this case we must have $\alpha\searrow 0$. Secondly, it does not exhibits the correct asymptotic behaviour as $s$ goes to $1$.
\end{oss}

\section{Proof of Theorem \ref{teo:teor valori ammissibili}}
\label{sec:5}

Let $s\in(0,1/2]$ and $\{Q_k\}_{k\in\mathbb{N}}\subset\mathbb{R}^2$ be the sequence of open squares $Q_k=(-k,k)^2$, with $k\in\mathbb{N}\setminus\{0,1\}$. We introduce the one-dimensional set 
$$
{\Sigma}=\bigcup\limits_{i\in\mathbb{Z}}{\Sigma}^{(i)},\qquad \mbox{ where }\, \Sigma^{(i)}:=\bigcup\limits_{i\in\mathbb{Z}}\left\{(x,i)\in\mathbb{R}^2\,:\,|x|\ge1\right\},
$$ 
and then define, for every fixed $k\in\mathbb{N}\setminus\{0,1\}$, the ``cracked'' square $\widetilde{Q}_k=Q_k\setminus{\Sigma}$ (see Figure \ref{fig:cracker}).\par 
	First of all, we observe that 
	\[
	r_{\widetilde{Q}_k}=\frac{\sqrt{5}}{2},\qquad\mbox{ for every } k\ge 2.
	\]
Thus, if we can show that 
	\begin{equation}
	\label{l1s va a zero}
		\lim_{k\to \infty}\lambda_1^s(\widetilde{Q}_k)=0,
	\end{equation}
	we would automatically get the desired counter-example. We will obtain \eqref{l1s va a zero} by proving that
	\begin{equation}
	\label{denti}
	\lambda_1^s(\widetilde{Q}_k)=\lambda_1^s(Q_k),\qquad \mbox{ for every } k\ge 2.
	\end{equation}
Indeed, if this were true, we would have 
	\[
	\lim_{k\to\infty}\lambda_1^s(\widetilde{Q}_k)=\lim_{k\to\infty}\lambda_1^s(Q_k)=\lim_{k\to\infty} k^{-2\,s}\,\lambda^s_1(Q_1)=0,
	\]
	by the scale properties of $\lambda^1_s$. This would prove \eqref{l1s va a zero}, as claimed.
\vskip.2cm\noindent
We are thus left with proving \eqref{denti}.
We already know that
\[
\lambda_1^s(\widetilde{Q}_k)\ge\lambda_1^s(Q_k),
\]
thanks to the fact that $\lambda_1^s$ is monotone with respect to set inclusion.	
	In the remaining part of the proof, we focus our attention in proving the opposite inequality.
	\vskip.2cm\noindent
	\begin{figure}
	\includegraphics[scale=.3]{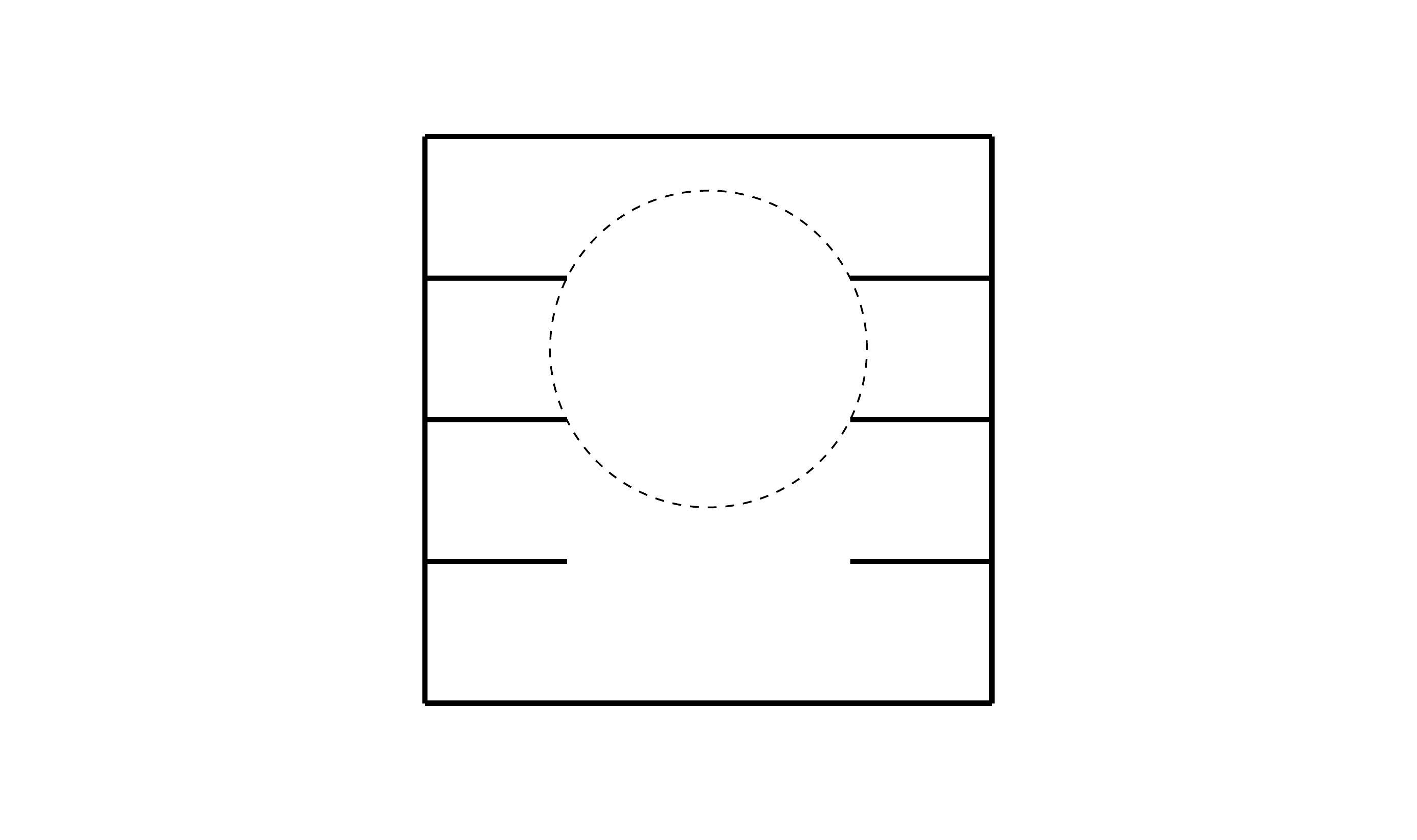}
	\caption{The set $\widetilde{Q}_k$ with $k=2$. In dashed line, a disk of maximal radius.}
	\label{fig:cracker}
	\end{figure}
At this aim, for every $n\in\mathbb{N}\setminus\{0\}$ we introduce the neighborhoods 
\[
\Sigma_{k,n}^{(i)}=\left\{x\in\mathbb{R}^2\, :\, \mathrm{dist}(x,\Sigma^{(i)}\cap Q_k)\le \frac{1}{n}\right\},	\qquad \mbox{ for } i\in\{-(k-1),\dots,k-1\},
\]
and consider a sequence of cut-off functions $\{\varphi_n^{(i)}\}_{n\in\mathbb{N}\setminus\{0\}}\subset C^\infty_0(\Sigma^{(i)}_{k,2\,n})$ such that
\[
0\le \varphi^{(i)}_n\le 1,\qquad \varphi^{(i)}_n\equiv 1 \mbox{ on } \Sigma^{(i)}_{k,4\,n},\qquad |\nabla \varphi^{(i)}_n(x)|\le C\,n,
\]	
for some constant $C>0$, independent of $n$. Observe that by construction we have
\[
\mathrm{spt}(\varphi^{(i)}_n)\cap \mathrm{spt}(\varphi^{(j)}_n)=\emptyset,\qquad \mbox{ for } i\not=j,
\]
\begin{figure}
	\includegraphics[scale=.3]{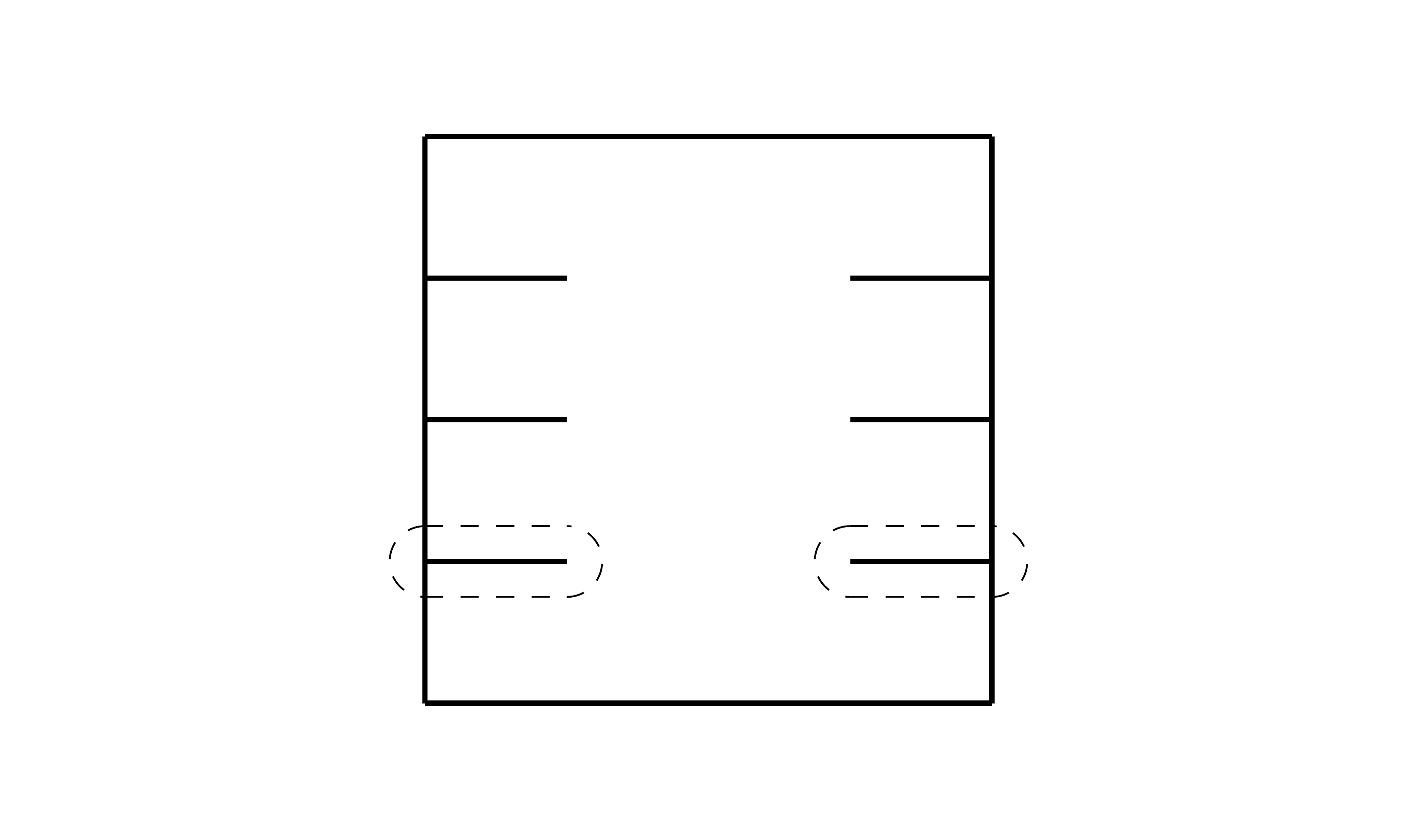}
	\caption{The dashed line encloses one of the set $\Sigma_{k,n}^{(i)}$.}
	\label{fig:cracker2}
	\end{figure}	
By using an interpolation inequality (see \cite[Corollary 2.2]{BPS}) and the properties of the cut-off functions, we can estimate the energy of each $\varphi_n^{(i)}$ as follows 
	\[
	\begin{split}
		[\varphi^{(i)}_n]^2_{W^{s,2}(\mathbb{R}^2)}&\le C\,\left(\int_{\Sigma_{k,2\,n}^{(i)}}|\varphi^{(i)}_n|^2\,dx\right)^{1-s}\,\left(\int_{\Sigma_{k,2\,n}^{(i)}}|\nabla\varphi^{(i)}_n|^2\,dx\right)^s\\
		&\le C\,|\Sigma_{k,2\,n}^{(i)}|^{1-s}\,|\Sigma_{k,2\,n}^{(i)}|^s\,n^{2\,s}\le C\,n^{2\,s-1},
	\end{split}
    \]
 for a constant $C>0$ independent\footnote{Observe that such a constant depend on $k$, through the length of the set $\Sigma^{(i)}\cap Q_k$. However this is not a problem, since in this part $k$ is now fixed.} of $n$. 
In particular, for every $i\in\{-(k-1),\dots,k-1\}$ we have
\begin{equation}
\label{stima1}
	\lim_{n\to+\infty}[\varphi^{(i)}_n]^2_{W^{s,2}(\mathbb{R}^2)}=0,\qquad\mbox{ if }0<s<\frac{1}{2},
\end{equation}
while
\begin{equation}\label{stima2}
	\sup_{n\ge 1}\,[\varphi^{(i)}_n]^2_{W^{s,2}(\mathbb{R}^2)}\le C,\qquad\mbox{ if }s=\frac{1}{2}.
\end{equation}
From now on, for ease of notation, we denote \[
	\Phi_{k,n}=\displaystyle\sum\limits_{i=-(k-1)}^{k-1}\varphi^{(i)}_n\in C^\infty_0(\mathbb{R}^2).
\]
Due to the different behaviours \eqref{stima1} and \eqref{stima2}, we need to consider the cases $0<s<1/2$ and $s=1/2$ separately.\par 
	\vskip.2cm\noindent
	{\bf Case $0<s<1/2$.}
For every $u\in C^\infty_0(Q_k)\setminus\{0\}$, we simply take 
\[
u_{n}=\left(1-\Phi_{k,n}\right)\,u,
\] 
and observe that $u_{n}\in C^\infty_0(\widetilde{Q}_k)$ for every $n\in\mathbb{N}\setminus\{0\}$. 
	Since each $u_n$ is admissible for the problem \eqref{def l1s}, we get
	\begin{equation}
	\label{stimetta}
	\sqrt{\lambda_1^s(\widetilde{Q}_k)}\le\frac{[u_{n}]_{W^{s,2}(\mathbb{R}^2)}}{\|u_n\|_{L^2(\widetilde{Q}_k)}}\le\frac{[u]_{W^{s,2}(\mathbb{R}^2)}+\|u\|_{L^\infty(\mathbb{R}^2)}\,\left[1-\Phi_{k,n}\right]_{W^{s,2}(\mathbb{R}^2)}}{\|u\,(1-\Phi_{k,n})\|_{L^2(\widetilde{Q}_k)}},
	\end{equation}
	where in the last inequality we have used the Leibniz--type rule \eqref{leibniz} and the fact $|1-\Phi_{k,n}|\le 1$. We now observe that 
	\[
	\lim_{n\to\infty} \|u\,(1-\Phi_{k,n})\|_{L^2(\widetilde{Q}_k)}=\|u\|_{L^2(Q_k)}, 
	\]
which follows from a standard application of the Lebesgue Dominated Convergence Theorem, together with the properties of $\Phi_{k,n}$. Moreover, it holds
\[
\lim\limits_{n\to\infty}\left[1-\Phi_{k,n}\right]_{W^{s,2}(\mathbb{R}^2)}=0.
\]
This simply follows by using the definition of $\Phi_{k,n}$, the triangle inequality and \eqref{stima1}. By using these two limits in \eqref{stimetta}, we get
	\[
	\sqrt{\lambda_1^s(\widetilde{Q}_k)}\le 
	\lim_{n\to\infty}\frac{[u]_{W^{s,2}(\mathbb{R}^2)}+\|u\|_{L^\infty(\mathbb{R}^2)}\,\left[1-\Phi_{k,n}\right]_{W^{s,2}(\mathbb{R}^2)}}{\|u\,(1-\Phi_{k,n})\|_{L^2(\widetilde{Q}_k)}}=\dfrac{[u]_{W^{s,2}(\mathbb{R}^2)}}{\|u\|_{L^2(Q_k)}}.
	\]
By arbitrariness of $u\in C^\infty_0(Q_k)\setminus\{0\}$, we get
	\[
	\lambda^s_1(\widetilde{Q}_k)\le \lambda_1^s(Q_k).
	\] 
and thus the desired conclusion \eqref{denti}.
\vskip.2cm\noindent
	{\bf Borderline case $s=1/2$.} This is more delicate, we can not use directly the sequence $\{\Phi_{k,n}\}_{n\in\mathbb{N}\setminus\{0\}}$ to construct an approximation of $u\in C^\infty_0(Q_k)$. Indeed, by owing to \eqref{stima2}, we can now guarantee that $\{\Phi_{k,n}\}_{n\in\mathbb{N}\setminus\{0\}}$ only converges weakly to $0$ in $W^{1/2,2}(\mathbb{R}^2)$, up to a subsequence.
\par
In order to ``boost'' such a sequence, we make a suitable application of {\it Mazur Lemma} (see for example \cite[Theorem 2.13]{LL}). More precisely, we define the sequence $\{F_{k,n}\}_{n\in\mathbb{N}\setminus\{0\}}\subset L^2(\mathbb{R}^2\times\mathbb{R}^2)$, given by
	\[
	F_{k,n}(x,y)=\frac{\Phi_{k,n}(x)-\Phi_{k,n}(y)}{|x-y|^{1+\frac{1}{2}}}.
	\]
	By construction, we have that 
	\[
	\|F_{k,n}\|_{L^2(\mathbb{R}^2\times\mathbb{R}^2)}=[\Phi_{k,n}]_{W^{\frac{1}{2},2}(\mathbb{R}^2)}\le C,
	\]
and $\{F_{k,n}\}_{n\in\mathbb{R}^2}$ converges weakly to $0$ in $L^2(\mathbb{R}^2\times\mathbb{R}^2)$, up to a subsequence.	
		Thanks to Mazur Lemma, we can enforce this weak convergence to the strong one, by passing to a sequence of convex combinations. More precisely, we know that for every $n\in\mathbb{N}\setminus\{0\}$ there exists 
	\[
\big\{t_\ell(n)\big\}_{\ell=1}^n\subset[0,1],\qquad \mbox{ such that }
\qquad	\sum_{\ell=1}^n t_\ell(n)=1,
	\]
and such that the new sequence made of convex combinations
	\[
	\widetilde{F}_{k,n}(x,y)=\sum_{\ell=1}^n t_\ell(n)\,F_{k,\ell}(x,y),
	\]
strongly converges in $L^2(\mathbb{R}^2\times\mathbb{R}^2)$ to $0$. Observe that by construction we have
\[
\begin{split}
\|\widetilde{F}_{k,n}\|^2_{L^2(\mathbb{R}^2\times\mathbb{R}^2)}&=\left\|\sum_{\ell=1}^n t_\ell(n)\,F_{k,\ell}\right\|^2_{L^2(\mathbb{R}^2\times\mathbb{R}^2)}\\
&=\iint_{\mathbb{R}^2\times\mathbb{R}^2} \left|\sum_{\ell=1}^n t_\ell(n)\frac{\Phi_{k,\ell}(x)-\Phi_{k,\ell}(y)}{|x-y|^{1+\frac{1}{2}}}\right|^2\,dx\,dy\\
&=\iint_{\mathbb{R}^2\times\mathbb{R}^2} \frac{\left|\sum_{\ell=1}^n t_\ell(n)\,\Phi_{k,\ell}(x)-\sum_{\ell=1}^n t_\ell(n)\,\Phi_{k,\ell}(y)\right|^2}{|x-y|^{3}}\,dx\,dy.
\end{split}
\]
Thus, if we set 
\[
\widetilde\Phi_{k,n}=\sum_{\ell=1}^n t_\ell(n)\,\Phi_{k,\ell}\in C^\infty_0(\mathbb{R}^2),
\]
the previous observations give that
\begin{equation}
\label{uffa!}
\lim_{n\to\infty} [\widetilde\Phi_{k,n}]_{W^{\frac{1}{2},2}(\mathbb{R}^2)}=\lim_{n\to\infty}\|\widetilde{F}_{k,n}\|^2_{L^2(\mathbb{R}^2\times\mathbb{R}^2)}=0.
\end{equation}
We take as in the previous case $u\in C^\infty_0(Q_k)\setminus\{0\}$. In order to approximate $u$ with functions compactly supported in $\widetilde{Q}_k$, we now define 
	\[
	\widetilde{u}_n=(1-\widetilde{\Phi}_{k,n})\,u.
	\]
We observe that this function belongs to $C^\infty_0(\widetilde Q_k)$. Indeed, observe that 
\[
\Phi_{k,\ell}(x)=1,\qquad \mbox{ for every } x\in \Sigma^{(i)}_{k,4\,\ell},\, i\in\{-(k-1),\dots,k-1\} \mbox{ and }\ell\in\{1,\dots,n\},  
\]
thus in particular
\[
\widetilde\Phi_{k,n}(x)=\sum_{\ell=1}^n t_\ell(n)\,\Phi_{k,\ell}(x)=\sum_{\ell=1}^n t_\ell(n)=1,\qquad \mbox{ for every }x\in \Sigma^{(i)}_{k,4\,n},\, i\in\{-(k-1),\dots,k-1\}, 
\]
thanks to the fact that 
\[
\Sigma^{(i)}_{k,4\,n}\subset\Sigma^{(i)}_{k,4\,\ell},\qquad \mbox{ for }\ell\in\{1,\dots,n\}.
\]
Clearly, we still have 
\begin{equation}
\label{mobbasta}
|1-\widetilde{\Phi}_{k,n}|\le 1\qquad \mbox{ and }\qquad \lim_{n\to\infty} \|\widetilde{u}_n\|_{L^2(\widetilde Q_k)}=\|u\|_{L^2(Q_k)}.
\end{equation}
We can now use $\widetilde u_n$ as a competitor for the variational problem defining $\lambda_1^s(\widetilde Q_k)$ and proceed exactly as in the case $0<s<1/2$, by using \eqref{uffa!} and \eqref{mobbasta}. This finally concludes the proof.
\begin{oss}
\label{oss:infinitecomb}
With the notation above, we obtain in particular that the {\it infinite complement comb} $\Theta:=\mathbb{R}^2\setminus \Sigma$ is an on open simply connected such that
\[
r_{\Theta}=\frac{\sqrt{5}}{2}\qquad \mbox{ and }\qquad \lambda_1^s(\Theta)=0,\ \mbox{ for } 0<s\le\frac{1}{2}.
\]
Indeed, by domain monotonicity and \eqref{l1s va a zero}, we have
\[
0\le \lambda_1^s(\Theta)\le \lim_{k\to\infty} \lambda_1^s(\widetilde{Q}_k)=0.
\]
\end{oss}

\section{Some consequences}
\label{sec:6}

We highlight in this section some consequences of our main result, by starting with a fractional analogue of property \eqref{equivalence} seen in the Introduction.

\begin{coro}
Let  $\Omega \subset\mathbb{R}^2$ be an open simply connected set. Then we have:
\begin{itemize}
\item for $1/2<s<1$
\[
\lambda_1^s(\Omega)>0 \qquad \Longleftrightarrow \qquad r_\Omega<+\infty;
\]
\item for $0<s\le 1/2$
\[
\lambda_1^s(\Omega)>0 \qquad \Longrightarrow \qquad r_\Omega<+\infty,
\]
but
\[
r_\Omega<+\infty \qquad  \not\Longrightarrow \qquad \lambda_1^s(\Omega)>0.
\]
\end{itemize}
\end{coro}
\begin{proof}
Let $0<s<1$ and assume that $\lambda_1^s(\Omega)>0$. Let $r>0$ be such that there exists $x_0\in \Omega$ with $B_r(x_0)\subset\Omega$. By using the monotonicity of $\lambda_1^s$ with respect to set inclusion, we get
\[
\lambda_1^s(\Omega)\le \lambda_1^s(B_r(x_0))=\frac{\lambda_1^s(B_1)}{r^{2\,s}}.
\]
The previous estimate gives 
\[
r<\left(\frac{\lambda_1^s(B_1)}{\lambda_1^s(\Omega)}\right)^\frac{1}{2\,s}.
\]
By taking the supremum over admissible $r$, we get $r_\Omega<+\infty$ by definition of inradius.
\par
For the converse implication in the case $s>1/2$, it is sufficient to apply Theorem \ref{teo:teorema principale}. Finally, by taking $\Theta$ as in Remark \ref{oss:infinitecomb}, we get an open set with finite inradius, but vanishing $\lambda_1^s$ for $0<s\le 1/2$.
\end{proof}
In turn, the previous result permits to compare two different Sobolev spaces, built up of functions ``vanishing at the boundary''. More precisely, let us denote by $\mathcal{D}^{s,2}_0(\Omega)$ the {\it completion} of $C^\infty_0(\Omega)$ with respect to the norm 
\[
u\mapsto [u]_{W^{s,2}(\mathbb{R}^N)},\qquad \mbox{ for every } u\in C^\infty_0(\Omega).
\]
Observe that this is indeed a norm on $C^\infty_0(\Omega)$. We refer to \cite{BGCV} for more details on this space.
By Theorem \ref{teo:teorema principale}, for an open simply connected set $\Omega\subset\mathbb{R}^2$, the two norms
\[
[u]_{W^{s,2}(\mathbb{R}^N)}\qquad \mbox{ and }\qquad \|u\|_{W^{s,2}(\mathbb{R}^N)},
\]
are equivalent on $C^\infty_0(\Omega)$, when $1/2<s<1$. Thus we easily get the following
\begin{coro}
Let $1/2<s<1$ and let $\Omega\subset \mathbb{R}^2$ be an open simply connected set, with finite inradius. Then
\[
\mathcal{D}^{s,2}_0(\Omega)=\widetilde W^{s,2}_0(\Omega).
\]
On the contrary, for $0<s\le 1/2$ and $\Theta$ the infinite complement comb of Remark \ref{oss:infinitecomb}, the two spaces
\[
\mathcal{D}^{s,2}_0(\Theta) \qquad \mbox{ and }\qquad \widetilde W^{s,2}_0(\Theta),
\] 
can not be identified with each other.
\end{coro}

We now show how our main result implies some fractional versions of the classical {\it Cheeger's inequality}, a fundamental result in Spectral Geometry. At this aim, for an open set $\Omega\subset\mathbb{R}^N$ we recall the definition of {\it Cheeger constant}
\[
h_1(\Omega)=\left\{\frac{P(E)}{|E|}\, :\, E\subset \Omega \mbox{ bounded and measurable with } |E|>0\right\},
\]
and {\it $s-$Cheeger constant} (for $0<s<1$)
\[
h_s(\Omega)=\left\{\frac{P_s(E)}{|E|}\, :\, E\subset \Omega \mbox{ bounded and measurable with } |E|>0\right\},
\]
see \cite{BLP} for some properties of this constant.
Here $P$ stands for the {\it perimeter} of a set in the sense of De Giorgi, while $P_s$ is the $s-$perimeter of a set, defined by
\[
P_s(E)=[1_E]_{W^{s,1}(\mathbb{R}^N)}=\iint_{\mathbb{R}^N\times\mathbb{R}^N} \frac{|1_E(x)-1_E(y)|}{|x-y|^{N+s}}\,dx\,dy,
\]
for any measurable set $E\subset \mathbb{R}^N$.
Then we have the following
\begin{coro}[Fractional Cheeger inequality]
\label{coro:cheeger}
Let $1/2<s<1$ and let $\Omega \subset\mathbb{R}^2$ be an open simply connected set, with finite inradius. Then we have
\[
\lambda_1^s(\Omega)\ge \mathcal{C}_s\left(\frac{h_1(\Omega)}{2}\right)^{2\,s},
\]
and
\[
\lambda_1^s(\Omega)\ge \mathcal{C}_s\,\left(\frac{\pi}{P_s(B_1)}\,h_s(\Omega)\right)^2.
\]
where $\mathcal{C}_s$ is the same constant as in Theorem \ref{teo:teorema principale}.
\end{coro}
\begin{proof}
Let $r<r_\Omega$, by definition of inradius there exists a disk $B_r(x_0)\subset \Omega$. By using this disk as a competitor for the minimization problem defining $h_1(\Omega)$, we get
\[
h_1(\Omega)\le \frac{2\,\pi\,r}{\pi\,r^2}=\frac{2}{r}.
\]
By taking the supremum over admissible $r$, we get
\[
h_1(\Omega)\le \frac{2}{r_\Omega}.
\]
By raising to the power $2\,s$ and using Theorem \ref{teo:teorema principale}, we get the first inequality. The second one can be obtained in exactly the same way.
\end{proof}
Finally, we have the following result, which permits to compare $\lambda_1^s(\Omega)$ and $\lambda_1(\Omega)$, for simply connected sets in the plane. We refer to \cite[Theorem 6.1]{BS} and \cite[Theorem 4.5]{CS} for a similar result in general dimension $N\ge 2$, under stronger regularity assumptions on the sets.
\begin{coro}[Comparison of eigenvalues]
Let $1/2<s<1$ and let $\Omega \subset\mathbb{R}^2$ be an open simply connected set, with finite inradius. Then we have
\begin{equation}
\label{eigencomp}
\alpha_s\,\Big(\lambda_1(\Omega)\Big)^s\le \lambda_1^s(\Omega)\le \beta_s\,\Big(\lambda_1(\Omega)\Big)^s,
\end{equation}
where $\alpha_s,\beta_s$ are two positive constants depending on $s$ only, such that 
\[
\alpha_s\sim \left(s-\frac{1}{2}\right),\ \mbox{ for } s\searrow \frac{1}{2},\qquad \mbox{ and }\qquad 
\alpha_s\sim \frac{1}{1-s},\ \mbox{ for } s\nearrow 1,
\]
\[
\beta_s\sim \frac{1}{1-s},\ \mbox{ for } s\nearrow 1.
\]
\end{coro}
\begin{proof}
The upper bound follows directly from the general result of \cite[Theorem 6.1]{BS}, see equation (6.1) there. From this reference, we can also extract a value for the constant $\beta_s$, which is given by
\[
\beta_s=\frac{4^{1-s}}{s\,(1-s)}\,\pi.
\]
For the lower bound, the proof is similar to that of Corollary \ref{coro:cheeger}, it is sufficient to join the estimate
\[
\lambda_1(\Omega)\le \frac{\lambda_1(B_1)}{r_\Omega^2},
\]
with Theorem \ref{teo:teorema principale}. This gives the claimed estimate, with constant 
\[
\alpha_s=\frac{\mathcal{C}_s}{(\lambda_1(B_1))^s},
\]
and $\mathcal{C}_s$ is the same as in \eqref{MHs}.
\end{proof}
\begin{oss}
The lower bound in estimate \eqref{eigencomp} degenerates as $s$ approaches $1/2$. This behaviour is optimal: indeed, observe that for the set $\Theta$ of Remark \ref{oss:infinitecomb} we have 
\[
\lambda_1(\Theta)>0\qquad \mbox{ and }\qquad \lambda_1^s(\Theta)=0,\ \mbox{ for } 0<s\le \frac{1}{2}.
\]
The first fact follows from \cite[Theorem 1, Section 15.4.2]{Maz}, for example. 
Thus the lower bound can not hold for this range of values.
\end{oss}

\appendix

\section{A one-dimensional Poincar\'e inequality}\label{appendice}
In what follows, we introduce the following norm on the one-dimensional torus $\mathbb{S}^1=\mathbb{R}/{(2\,\pi\,\mathbb{Z})}$
\[
|\alpha|_{\mathbb{S}^1}:=\min\limits_{k\in\mathbb{Z}}|\alpha+2\,k\,\pi|,\qquad \mbox{ for every } \alpha\in\mathbb{R}.
\]
We observe in particular for $\alpha\in[0,2\,\pi]$ this quantity is given by
\begin{equation}
\label{normaT}
|\alpha|_{\mathbb{S}^1}=\left\{\begin{array}{cc}
\alpha, & \mbox{ if } 0\le \alpha\le \pi,\\
2\,\pi-\alpha, & \mbox{ if } \pi<\alpha\le 2\,\pi.
\end{array}
\right.
\end{equation}
\begin{lemma}
\label{lm:equivalenza}
	We have
	\[
	\frac{2}{\pi}\,|\theta-\varphi|_{\mathbb{S}^1}\le|e^{i\,\theta}-e^{i\,\varphi}|\le |\theta-\varphi|_{\mathbb{S}^1},\qquad \mbox{ for every }\theta,\varphi\in \mathbb{R}.
	\] 
Moreover, both inequalities are sharp.	
\end{lemma}
\begin{proof}
We first observe that we can write
\begin{equation}
\label{complessi}
\begin{split}
|e^{i\,\theta}-e^{i\,\varphi}|&=|e^{i\,\varphi}|\,|e^{i\,(\theta-\varphi)}-1|\\
&=|e^{i\,(\theta-\varphi)}-1|\\
&=\sqrt{(1-\cos (\theta-\varphi))^2+\sin^2(\theta-\varphi)}=2\,\left|\sin\left(\frac{\theta-\varphi}{2}\right)\right|,
\end{split}
\end{equation}
thanks to standard trigonometric formulas. In order to conclude the proof, it is sufficient to prove that 
\begin{equation}
\label{ridotta}
\frac{2}{\pi}\,|\alpha|_{\mathbb{S}^1}\le 2\,\left|\sin\left(\frac{\alpha}{2}\right)\right|\le |\alpha|_{\mathbb{S}^1},\qquad \mbox{ for every } \alpha\in\mathbb{R}.
\end{equation}
It is easily see that both functions
\[
\alpha\mapsto |\alpha|_{\mathbb{S}^1} \qquad \mbox{ and }\qquad \alpha\mapsto \left|\sin\left(\frac{\alpha}{2}\right)\right|,
\]
are $2\,\pi-$periodic, thus is it sufficient to prove \eqref{ridotta} for $\alpha\in[0,2\,\pi]$. We thus seek for the maximum and the minimum on $[0,2\,\pi]$ of the function
\[
\alpha\mapsto 2\,\frac{|\sin(\alpha/2)|}{|\alpha|_{\mathbb{S}^1}},
\]	
extended by continuity to the whole interval. 
By keeping in mind \eqref{normaT}, on $[0,2\,\pi]$ this function can be rewritten as
\[
\alpha\mapsto \left\{\begin{array}{cc}
2\,\dfrac{\sin(\alpha/2)}{\alpha},& \mbox{ if } 0\le \alpha\le \pi,\\
&\\
2\,\dfrac{\sin(\alpha/2)}{2\,\pi-\alpha},& \mbox{ if } \pi\le \alpha\le 2\,\pi,
\end{array}
\right.=\left\{\begin{array}{cc}
\dfrac{\sin(\alpha/2)}{\alpha/2},& \mbox{ if } 0\le \alpha\le \pi,\\
&\\
\dfrac{\sin(\pi-\alpha/2)}{\pi-\alpha/2},& \mbox{ if } \pi\le \alpha\le 2\,\pi.
\end{array}
\right.
\]
By recalling that the {\it sinc function} $t\mapsto (\sin t)/t$ is monotone decreasing on the interval $[0,\pi/2]$, in light of the above discussion we now easily obtain
\[
\frac{2}{\pi}\,\le 2\,\frac{|\sin(\alpha/2)|}{|\alpha|_{\mathbb{S}^1}}\le 1.
\]
This gives \eqref{ridotta}, thus concluding the proof. 
\end{proof}
The main result of this appendix is the following one-dimensional Poincar\'e inequality, for periodic functions vanishing at a point. The result is probably well-known, but as always we want to pay particular attention to the dependence of the constant on the parameter $s$. For $T>0$, we define the one-dimensional torus $\mathbb{S}^1_T=\mathbb{R}/{(T\,\mathbb{Z})}$, endowed with the norm
\[
|\theta-\varphi|_{\mathbb{S}_T^1}=\min\limits_{k\in\mathbb{Z}}|\theta-\varphi+k\,T|,\qquad \mbox{ for }\theta,\varphi\in \mathbb{R}.
\]
\begin{prop}
\label{prop poincare angolare}
	Let $1/2<s<1$ and $T>0$. Let $\theta_0\in[0,T]$, there exists a constant $\mu_s>0$ depending on $s$ only such that for every $w\in C^1(\mathbb{R})$ which is $T-$periodic and vanishing at $\theta_0$, we have
	\begin{equation}
	\label{1dpoin}
\mu_s\,\left(\frac{2\,\pi}{T}\right)^{2\,s}\,\int_0^T |w(\theta)|^2\,d\theta\le \iint_{[0,T]\times[0,T]} \frac{|w(\theta)-w(\varphi)|^2}{|\theta-\varphi|_{\mathbb{S}_T^1}^{1+2\,s}}\,d\theta\,d\varphi,
	\end{equation}
	Moreover, the constant $\mu_s$ has the following asymptotic behaviours
	\[
	\mu_s\sim \left(s-\frac{1}{2}\right),\ \mbox{ for } s\searrow \frac{1}{2}
\qquad	\mbox{ and }\qquad \mu_s\sim \frac{1}{1-s},\ \mbox{ for } s\nearrow 1.
	\]
\end{prop}
\begin{proof} Without loss of generality, we can assume that $\theta_0=0$ and $T=2\,\pi$. Thus, in this case we have $|\cdot|_{\mathbb{S}_{2\pi}^1}=|\cdot|_{\mathbb{S}^1}$, with the notation of Lemma \ref{lm:equivalenza}.\par
Thanks to the periodicity of $w$, we can expand it in Fourier series, i.e. we can write 
	\[
	w(\theta)=\sum\limits_{n\in\mathbb{Z}}\widehat{w}(n)\,e^{in\theta},\qquad\mbox{where }\quad \widehat{w}(n)=\frac{1}{2\,\pi}\,\int_{0}^{2\,\pi}w(\theta)\,e^{-i\,n\,\theta}\,d\theta.
	\]
The series is uniformly converging, thanks to the assumption on $w$.
	We will achieve the claimed result by joining the following two estimates
	\begin{equation}\label{disug seminorma fourier}
	\iint_{[0,2\,\pi]\times[0,2\,\pi]} \frac{|w(\theta)-w(\varphi)|^2}{|\theta-\varphi|_{\mathbb{S}^1}^{1+2\,s}}\,d\theta\,d\varphi\ge  C_{1,s}\,\sum\limits_{n\in\mathbb{Z}}|n|^{2\,s}\,|\widehat{w}(n)|^2,
	\end{equation}
and	
\begin{equation}\label{disug norma l2 fourier}
	\int_{0}^{2\,\pi} |w(\theta)|^2\,d\theta\le C_{2,s}\,\sum\limits_{n\in\mathbb{Z}}|n|^{2\,s}\,|\widehat{w}(n)|^2,
	\end{equation}
that we prove separately. This would give \eqref{1dpoin}, with constant $\mu_s=C_{1,s}/C_{2,s}$. In the last part of the proof, we will then prove that such a constant has the claimed asymptotics.
\vskip.2cm\noindent
{\bf Proof of \eqref{disug seminorma fourier}}.
We proceed similarly as in the proof of \cite[Proposition 3.4]{DPV}, with suitable adaptations. The latter deals with $W^{s,2}$ functions on $\mathbb{R}$ and their Fourier transform.
\par
First of all, we rewrite the Gagliardo-Slobodecki\u{\i} seminorm as follows:
let us apply the change of variable $h=\varphi-\theta$, so to get
\[
	\begin{split}
	\int_0^{2\,\pi}\int_0^{2\,\pi}\frac{|w(\varphi)-w(\theta)|^2}{|\varphi-\theta|_{\mathbb{S}^1}^{1+2\,s}}\,d\theta\,d\varphi&=\int_0^{2\,\pi}\int_{-\theta}^{2\,\pi-\theta}\frac{|w(\theta+h)-w(\theta)|^2}{|h|_{\mathbb{S}^1}^{1+2\,s}}\,d\theta\,dh\\
	&=\int_{0}^{2\,\pi}\int_{-\theta}^{0}\frac{|w(\theta+h)-w(\theta)|^2}{|h|_{\mathbb{S}^1}^{1+2\,s}}\,d\theta\,dh\\
	&+\int_{0}^{2\,\pi}\int_{0}^{2\,\pi}\frac{|w(\theta+h)-w(\theta)|^2}{|h|_{\mathbb{S}^1}^{1+2\,s}}\,d\theta\,dh\\
	&-\int_{0}^{2\,\pi}\int_{2\,\pi-\theta}^{2\,\pi}\frac{|w(\theta+h)-w(\theta)|^2}{|h|_{\mathbb{S}^1}^{1+2\,s}}\,d\theta\,d h.
	\end{split}
	\]
	On the third integral, we can use that the integrand is $2\,\pi-$periodic, thus we get
\[
\begin{split}
\int_{0}^{2\,\pi}\int_{2\,\pi-\theta}^{2\,\pi}\frac{|w(\theta+h)-w(\theta)|^2}{|h|_{\mathbb{S}^1}^{1+2\,s}}\,d\theta\,dh&=\int_{0}^{2\,\pi}\int_{2\,\pi-\theta}^{2\,\pi}\frac{|w(\theta+h-2\,\pi)-w(\theta)|^2}{|h-2\,\pi|_{\mathbb{S}^1}^{1+2\,s}}\,d\theta\,dh\\
&=\int_{0}^{2\,\pi}\int_{-\theta}^{0}\frac{|w(\theta+\eta)-w(\theta)|^2}{|\eta|_{\mathbb{S}^1}^{1+2\,s}}\,d\theta\,d\eta.
\end{split}
\]
This finally permits to infer that
\[
\begin{split}
\int_0^{2\,\pi}\int_0^{2\,\pi}\frac{|w(\varphi)-w(\theta)|^2}{|\varphi-\theta|_{\mathbb{S}^1}^{1+2\,s}}\,d\theta\,d\varphi&=\int_{0}^{2\,\pi}\int_{0}^{2\,\pi}\frac{|w(\theta+h)-w(\theta)|^2}{|h|_{\mathbb{S}^1}^{1+2\,s}}\,d\theta\,dh.
	\end{split}
	\]
	By recalling \eqref{normaT}, we can conclude that
	\begin{equation}\label{seminorma con h}
	\begin{split}
	\int_{0}^{2\,\pi}\int_{0}^{2\,\pi}\frac{|w(\theta+h)-w(\theta)|^2}{|h|_{\mathbb{S}^1}^{1+2\,s}}\,d\theta\,dh&= \int_{0}^\pi\frac{1}{h^{1+2\,s}}\left(\int_{0}^{2\,\pi}|w(\theta+h)-w(\theta)|^2\,d\theta\right)\,dh\\
	&+\int_{\pi}^{2\,\pi}\frac{1}{(2\,\pi-h)^{1+2\,s}}\left(\int_{0}^{2\,\pi}|w(\theta+h)-w(\theta)|^2\,d\theta\right)\,dh.
	\end{split}
	\end{equation}
	Now, for every $h$ we denote by $w_h(\theta)$ the translation $w_h(\theta)=w(\theta+h)$. 
	Thanks to the well-known properties of the Fourier coefficients, we have
	\begin{equation}\label{traslazione fourier}
	\begin{split}
	\widehat{w}_h(n)=e^{i\,h\,n}\,\widehat{w}(n),\qquad \mbox{ for every } n\in\mathbb{Z}.
	\end{split}
	\end{equation}
	By using Plancherel's identity in \eqref{seminorma con h} and then applying \eqref{traslazione fourier}, we finally obtain 
	\begin{equation}\label{seminorma con fourier}
	\begin{split}
	\int_{0}^{2\,\pi}\int_{0}^{2\,\pi}&\frac{|w(\theta+h)-w(\theta)|^2}{|h|_{\mathbb{S}^1}^{1+2\,s}}\,d\theta\,dh=\\
	&=2\,\pi\,\sum_{n\in\mathbb{Z}\setminus\{0\}}\left(\int_{0}^\pi\frac{|e^{i\,h\,n}-1|^2}{h^{1+2\,s}}\,dh+\int_{\pi}^{2\,\pi}\frac{|e^{i\,h\,n}-1|^2}{(2\,\pi-h)^{1+2\,s}}\,dh\right)\,|\widehat{w}(n)|^2.
	\end{split}
	\end{equation}
By recalling the identities \eqref{complessi}, we have
\[
|e^{i\,h\,n}-1|^2=2\,(1-\cos(h\,n)),
\]
and applying the change of variable $\tau=h\,n$ with $n\in\mathbb{Z}\setminus\{0\}$,
we can rewrite the first integral on the right-hand side of \eqref{seminorma con fourier} as
	\[
	\begin{split}
	\int_{0}^\pi\frac{|e^{i\,h\,n}-1|^2}{h^{1+2\,s}}\,dh&=2\,\int_{0}^\pi\frac{1-\cos(h\,n)}{h^{1+2\,s}}\,dh\\
	&=2\,\int_{0}^{\pi\,n}\frac{1-\cos \tau}{\left(\dfrac{\tau}{n}\right)^{1+2\,s}}\,\frac{d\tau}{n}\ge 2\,|n|^{2\,s}\,\int_{0}^{\pi} \frac{1-\cos\tau}{\tau^{2\,s}}\,\frac{d\tau}{\tau}.
	\end{split}
	\]
For the second integral, it is sufficient to observe that by periodicity
\[
\begin{split}
\int_{\pi}^{2\,\pi}\frac{|e^{i\,h\,n}-1|^2}{(2\,\pi-h)^{1+2\,s}}\,dh&=2\,\int_{\pi}^{2\,\pi}\frac{1-\cos(h\,n)}{(2\,\pi-h)^{1+2\,s}}\,dh\\
&=2\,\int_{\pi}^{2\,\pi}\frac{1-\cos(2\,\pi\,n-h\,n)}{(2\,\pi-h)^{1+2\,s}}\,dh\\
&=2\,\int_{0}^\pi\frac{1-\cos(h\,n)}{h^{1+2\,s}}\,dh\ge 2\,|n|^{2\,s}\,\int_{0}^{\pi} \frac{1-\cos\tau}{\tau^{2\,s}}\,\frac{d\tau}{\tau}.
\end{split}
\]	
Thus, from \eqref{seminorma con fourier} we get in particular
\[
\int_{0}^{2\,\pi}\int_{0}^{2\,\pi}\frac{|w(\theta+h)-w(\theta)|^2}{|h|_{\mathbb{S}^1}^{1+2\,s}}\,d\theta\,dh\ge 
	8\,\pi\,\left(\int_{0}^{\pi} \frac{1-\cos\tau}{\tau^{2\,s}}\,\frac{d\tau}{\tau}\right)\,\sum\limits_{n\in\mathbb{Z}}|n|^{2\,s}\,|\widehat{w}(n)|^2.
\]	
This finally proves \eqref{disug seminorma fourier}, with constant
\[
C_{1,s}=8\,\pi\,\int_{0}^{\pi} \frac{1-\cos\tau}{\tau^{2\,s}}\,\frac{d\tau}{\tau}.
\]
\noindent
{\bf Proof of} \eqref{disug norma l2 fourier}: from Plancherel's identity, we know that
	\begin{equation}\label{numero}
	\frac{1}{2\pi}\,\int_{0}^{2\,\pi}|w(\theta)|^2\,d\theta=\sum_{n\in \mathbb{Z}}|\widehat{w}(n)|^2.
	\end{equation}
By using the Fourier expansion for $w$ and the assumption $w(0)=w(2\,\pi)=0$, we can infer that
	\[
	0=w(0)=\sum\limits_{n\in\mathbb{Z}}\widehat{w}(n).
	\]
This in turn implies that
\[
	|\widehat{w}(0)|=\left|\sum\limits_{n\in\mathbb{Z}\setminus\{0\}}\widehat{w}(n)\right|\le \sum_{n\in\mathbb{Z}\setminus\{0\}} |\widehat{w}(n)|,
\]	and so we can obtain
	\begin{equation}\label{app1eq2}
	\begin{split}
	\sum_{n\in\mathbb{Z}}|\widehat{w}(n)|^2&=\sum_{n\in\mathbb{Z}\setminus\{0\}}|\widehat{w}(n)|^2+|\widehat{w}(0)|^2
	\le \sum_{n\in\mathbb{Z}\setminus\{0\}}|\widehat{w}(n)|^2+\left(\sum_{n\in\mathbb{Z}\setminus\{0\}}|\widehat{w}(n)|\right)^2.
	\end{split}
	\end{equation}
We now estimate the last term in \eqref{app1eq2} by using H\"older's inequality
	\[
	\sum\limits_{n\in\mathbb{Z}}|a_n|\,|b_n|\le\left(\sum_{n\in\mathbb{Z}}|a_n|^2\right)^\frac{1}{2}\,\left(\sum_{n\in\mathbb{Z}}|b_n|^2\right)^\frac{1}{2},
	\]
	with the choices $|a_n|=1/|n|^{s}$ and $|b_n|=|\widehat{w}(n)|\,|n|^{s}$.
This yields
\[
\begin{split}
\sum_{n\in\mathbb{Z}}|\widehat{w}(n)|^2&\le \sum_{n\in\mathbb{Z}\setminus\{0\}}|\widehat{w}(n)|^2+\left(\sum\limits_{n\in\mathbb{Z}\setminus\{0\}}\frac{1}{|n|^{2\,s}}\right)\,\left(\sum\limits_{n\in\mathbb{Z}\setminus\{0\}}|n|^{2\,s}\,|\widehat{w}(n)|^2\right)\\
&\le C\,\left(\sum\limits_{n\in\mathbb{Z}\setminus\{0\}}|n|^{2s}\,|\widehat{w}(n)|^2\right),
\end{split}
\]	
where we set
\[
C=1+2\,\sum_{n=1}^\infty \frac{1}{n^{2\,s}}.
\]
Observe that this is a finite quantity, thanks to the crucial assumption $s>1/2$. By using this estimate in \eqref{numero}, we then obtain the claimed inequality \eqref{disug norma l2 fourier}, with constant
\[
C_{2,s}=2\,\pi\,\left(1+2\,\sum_{n=1}^\infty \frac{1}{n^{2\,s}}\right).
\]
\vskip.2cm\noindent 
{\bf Asymptotic behaviour of the constant.}	 As we said, from the above discussion we get the inequality \eqref{1dpoin}, with $
\mu_s=C_{1,s}/C_{2,s}$.
It is easily seen that 
\[
\lim_{s\to \left(\frac{1}{2}\right)^+} C_{1,s}=8\,\pi\,\int_{0}^{\pi} \frac{1-\cos\tau}{\tau}\,\frac{d\tau}{\tau}<+\infty,
\]
while 
\[
\lim_{s\to \left(\frac{1}{2}\right)^+} (2\,s-1)\,C_{2,s}=2\,\pi\,
\lim_{s\to \left(\frac{1}{2}\right)^+}(2\,s-1)\,\left(1+2\,\sum_{n=1}^\infty \frac{1}{n^{2\,s}}\right)=4\,\pi,
\]
by using the fact that the Riemann zeta function has a simple pole with residue $1$ at $z=1$ (see \cite[Section 13.2.6]{Kr}). This proves that $\mu_s$ has the claimed asymptotic behaviour, as $s$ goes to $1/2$.
\par
As for the behaviour at $s\sim 1$, we observe that 
\[
\lim_{s\to 1^-} C_{2,s}=2\,\pi\,\left(1+2\,\sum_{n=1}^\infty \frac{1}{n^{2}}\right)=2\,\pi\,\left(1+\frac{\pi^2}{3}\right),
\] 
while
\[
\begin{split}
\lim_{s\to 1^-} (1-s)\,C_{1,s}&=8\,\pi\,\lim_{s\to 1^-}(1-s)\,\int_0^\pi\frac{1-\cos\tau}{\tau^{2\,s}}\,\frac{d\tau}{\tau}\\
&=4\,\pi\,\lim_{s\to 1^-}(1-s)\,\int_0^\pi\tau^{2-2\,s}\,\frac{d\tau}{\tau}\\
&-4\,\pi\,\lim_{s\to 1^-}(1-s)\,\int_0^\pi\frac{\int_0^\ell (\tau-\ell)^2\,\sin\ell\,d\ell}{\tau^{2\,s}}\,\frac{d\tau}{\tau}=2\,\pi,
\end{split}
\]
where we used the third order Taylor expansion
\[
f(\tau)=f(0)+f'(0)\,\tau+\frac{1}{2}\,f''(0)\,\tau^2+\frac{1}{2}\,\int_0^\tau f'''(\ell)\,(\tau-\ell)^2\,d\ell,
\]
for the cosine function. This eventually leads to the conclusion of the proof.
	\end{proof}

\end{document}